\documentclass[11pt, a4paper, reqno]{amsart}

\usepackage[dvipsnames]{xcolor}
\usepackage[backref=page, linktocpage]{hyperref}
\hypersetup{colorlinks=true,
	linkcolor=blue,
	citecolor=blue,
	urlcolor=blue}    
\usepackage{latexsym, amsfonts, amsmath, amssymb, amsthm, cite}
\usepackage{dsfont}
\usepackage{stmaryrd}
\usepackage{fullpage}
\usepackage{enumerate}
\usepackage{enumitem}
\usepackage{graphicx}

\newtheorem{theorem}{Theorem}[section]
\newtheorem{Lemma}[theorem]{Lemma}
\newtheorem{proposition}[theorem]{Proposition}

\newtheorem*{remark}{Remark}

\numberwithin{equation}{section}

\newcommand\eps{\varepsilon}

\renewcommand{\leq}{\leqslant}
\renewcommand{\geq}{\geqslant}
\renewcommand\Im{\operatorname{Im}} 

\allowdisplaybreaks[1]

\begin{document}
	
	\title[Subconvexity for Rankin Selberg L-Functions at Special Points]{Subconvexity for Rankin Selberg L-Functions at Special Points}

	\begin{abstract}
		Let $f$ and $g$ be normalized Hecke-Maass cusp forms for the full modular group having spectral parameters $t_f$ and $t_g$ respectively with $t_f,t_g\asymp T\rightarrow \infty $. In this paper we show that the Rankin Selberg $L$-function associated to the pair $(f,g)$ at the special points $t=\pm(t_f+t_g)$, satisfies the subconvex bound
        \begin{align*}
            L\left(\frac{1}{2}+it,f\otimes g\right)\ll_{\eps} T^{61/84+\eps}.
        \end{align*}
        Additionally at the points $t=\pm(t_f-t_g)\asymp T^{\nu}$ with $2/3+\eps<\nu\leq 1$ we show the subconvex bound 
        \begin{align*}
             L(1/2+it,f\otimes g)\ll_\eps {T^{7/12+\nu/8+\eps}}, \; \text{if }\; 2/3+\eps< \nu\leq 14/17,
        \end{align*}
        and
        \begin{align*}
            L(1/2+it,f\otimes g)\ll_\eps {T^{1/2+19\nu/84+\eps}}, \; \text{if }\; 14/17\leq \nu\leq 1.
        \end{align*}
        With the above results we are able to address the subconvexity problem in the spectral aspect for $GL(2)\times GL(2)$ Rankin Selberg $L$-functions when the parameters
        of both the forms vary under the additional challenge of a considerable amount conductor dropping occurring due to the special points in question.
	\end{abstract}

	\author{Sayan Ghosh}
	\address{Stat-Math Unit\\
		Indian Statistical Institute\\
		Kolkata\\
		700108\\
		India}
	\email{\href{mailto:email}{gsayan74@gmail.com}}

	
	\maketitle
	
	\setcounter{tocdepth}{2}
	
	\tableofcontents
	
	\section{Introduction}\label{intro}
    
    A far reaching and challenging problem in number theory is concerned with bounds on families of automorphic $L$-functions on the critical line. An automorphic $L$-function $L(s,f)$ is a complex valued function represented by a Dirichlet series and an Euler product, which is attached to an automorphic form $f$. It is said to be of degree $d$, if the degree of the Euler product equals $d$. The $L$-function $L(s,f)$ has a meromorphic continuation to the whole complex plane $\mathbb{C}$, and its completion satisfies a functional equation relating its value at $s$ to the value at $1-s$ of the $L$-function of the corresponding dual form. In order to estimate $L(s,f)$ on the critical line, a quantity known as the analytic conductor ($C(f,t)$) has been defined in literature (see Section 2 of \cite{IS00}). It measures the complexity of $L(s,f)$ depending on several of its parameters (such as the level of $f$, spectral parameters of $f$ or the continuous parameter $t$). An application of the Phragmen Lindeloff principle along with the functional equation implies the convexity bound, $L(1/2+it,f)\ll_{d,\eps}C(f,t)^{1/4+\eps}$. The subconvexity problem (ScP) aims to reduce the exponent $1/4$ by some positive amount, independent of $\eps$. Though, the far out of reach generalised Lindeloff Hypothesis (GLH) predicts that the exponent could be reduced to 0, obtaining bounds which are subconvex is still quite challenging.

    \subsection{History of the problem}We now recall a brief history of the subconvexity problem. For degree one $L$-functions, such as $\zeta(s)$ and Dirichlet $L$-functions $L(s,\chi)$, subconvexity is known due to  Weyl\cite{Weyl21} and Hardy-Littlewood in the $t$-aspect and due to Burgess \cite{Bur63} in the level aspect. For degree two $L$-functions subconvexity in $t$-aspect was first obtained by Good \cite{Go82}, by Duke-Friedlander-Iwaniec\cite{DFI93} in the level aspect, by Iwaniec \cite{Iw92} in the spectral aspect, by Jutial-Motohashi\cite{JM05}  in the $t$ and spectral aspect uniformly but away from the conductor dropping range and the problem has been solved in full generality (uniformity in all parameters) by Michel and Vekatesh\cite{MV10}. For degree three $L$-functions attached to self dual forms subconvex estimates in $t$-aspect were first obtained by Li\cite{Li11} and generalised to all $GL(3)$ forms by Munshi\cite{Mun15} by a novel ``delta-symbol" approach. In the spectral aspect for $GL(3)$ $L$-functions, such estimates were obtained by Blomer-Buttcane\cite{BB20} when the spectral parameters are restricted to ``generic" position. For higher degree $L$-functions, subconvex estimates have been relatively few and obtained mostly for $GL(2)\times GL(2)$ and $GL(3)\times GL(2)$ Rankin Selberg $L$-functions. Instances of such estimates, where one of the form is kept fixed, has been found notably in \cite{HM06}, \cite{Hu24}, \cite{JM06}, \cite{KMV02}, \cite{Kum25}, \cite{Li11}, \cite{LLY06}, \cite{MV10}, \cite{Mun22}, \cite{PY23}, \cite{Sar01} and \cite{Sha22}. Finally in a breakthrough paper, Nelson\cite{Nel23} has resolved the subconvexity problem in spectral aspect for higher rank groups, away from the conductor dropping range.
    
    \subsection{Motivation and statements of our results} While studying the above results on Rankin-Selberg $L$-functions, it is a natural question to ask what happens when both the forms vary. Results of such uniform nature has been sporadic and obtained mostly in the level aspect, as in \cite{HolM13} and \cite{Ye14}. We aim to address this question in the archimedean aspect (involving $t$ and the spectral parameters) in this paper. Added to that we also aim to explore the problem when the concerned family of $L$-functions exhibit ``conductor dropping phenomenon''. Subconvexity estimates in such cases are known to be quite difficult to obtain historically and have been rare in literature. Notable examples include Michel-Venkatesh\cite{MV10}, where they settled subconvexity for $GL(2)\times GL(2)$ $L$-functions in full generality when one of the forms is kept fixed.
    With all that in the background, our main results are the following:
    \begin{theorem}\label{Theorem}
        Let $f$ and $g$ be normalized Hecke-Maass cusp forms for the full modular group $\text{SL}_2(\mathbb{Z})$ with Laplacian eigenvalues $1/4+t_f^2$ and $1/4+t_g^2$ respectively. Assume that 
        \begin{align*}
            0\leq t_f,t_g\asymp T,\;\;\;\; T\rightarrow\infty.
        \end{align*}
        Then the Rankin-Selberg $L$-function $L(s,f\otimes g)$ at the special point $s=1/2+it,\; t=t_f+t_g$, satisfies the subconvex bound.
        \begin{align}\label{1.1}
            L\left(\frac{1}{2}+it,f\otimes g\right)\ll_{\eps} T^{61/84+\eps}.
        \end{align}
         Taking conjugates, the same bound also holds at the points $t=-t_f-t_g$.
    \end{theorem}
    \begin{theorem}\label{Theorem2}
        With the same premise of Theorem \ref{Theorem} and the added condition that 
        \begin{align*}
            t_f-t_g\asymp T^{\nu},\; \text{where }\; 2/3+\eps< \nu\leq 1,
        \end{align*} 
        the Rankin Selberg $L$-function $L(s,f\otimes g)$ at the point $s=1/2+it,$ $t=t_f-t_g$ satisfies the subconvex bound 
        \begin{align}
            L(1/2+it,f\otimes g)\ll_\eps {T^{7/12+\nu/8+\eps}}, \; \text{if }\; 2/3+\eps< \nu\leq 14/17, 
        \end{align}
        and 
        \begin{align}
            L(1/2+it,f\otimes g)\ll_\eps {T^{1/2+19\nu/84+\eps}}, \; \text{if }\; 14/17\leq \nu\leq 1.
        \end{align}
        Taking conjugates, the same bound also holds at the points $t=-t_f+t_g$.
    \end{theorem}
        The analytic conductor of $L(1/2+it,f\otimes g)$ is
        \begin{align*}
            C(f\otimes g,t)\asymp \prod_{\pm }\prod_{\pm}(1/2+|t\pm t_f\pm t_g|).
        \end{align*}
        Its size is approximately $T^4$ if $t\asymp T$ with $t$ being away from the points $\ t_f\pm t_g$ and drops to $T^{3}$ (resp. $T^2|t_f-t_g|\asymp T^{2+\nu}$) at points $t=t_f+t_g$ (resp. $t=t_f-t_g$), exhibiting a large drop in the size of the conductor. Reaching subconvexity for $L$-functions in conductor dropping ranges has traditionally posed quite a significant challenge to researchers with only a small number of results available (such as \cite{MV10}). With Theorems \ref{Theorem} and \ref{Theorem2} we are thus able to address the following objectives:
        \begin{itemize}
            \item Address the subconvexity problem for Rankin Selberg $L$-functions when both forms vary simultaneously, ensuring a high degree of conductor dropping taking place due to the special points in question adding to the results of Michel-Venkatesh  (\cite{MV10}) and Jutila-Motohashi (\cite{JM05}, \cite{JM06}) in the archimedean aspects.
            
            \item Considering the premise of Theorem \ref{Theorem2}, in the limiting scenario $\nu= 2/3+\eps$ the conductor is $T^{8/3+\eps}$ which is approximately two-thirds of the exponent $4$ in $T^4$. The same reduction (two-thirds) of the exponent occurs when passing from a generic $GL(3)$ $L$-function to the symmetric square $L$-function $L(1/2+it, \text{sym}^2f)$ of a $GL(2)$ form $f$ keeping $t\asymp T^\eps$. Thus Theorem \ref{Theorem2} can be regarded as an analogue for subconvex bounds of $L(1/2+it, \text{sym}^2f)$ for $t\asymp T^\eps$, in the higher rank setting. 
        \end{itemize}

    \begin{remark}
        Our approach is based on the delta symbol method pioneered by Munshi, marking the first time this technique has been applied to a problem of this nature. A brief outline of the proof has been provided in \S\S\S \ref{sketch}. We point out that here we are not aiming to obtain the best possible bounds for the families by the adopted technique. There might be room for improvement in the strength of our bounds by obtaining more precise estimates in Proposition \ref{prop_7.4} using stationary phase analysis. We only provide the detailed proof of Theorem \ref{Theorem}. The same methodology will also give the result of Theorem \ref{Theorem2}, with minor which are mentioned in \S \ref{Modifictions}.The implied constants in the above results depend additionally on $c_1,c_2$ where $c_1T\leq t_f,t_g\leq c_2T$. Throughout the proof even though we ignore the dependency of the implied constant on $c_1,c_2$, it should be clear when they occur.
    \end{remark}
    \begin{remark}
        Preliminary analysis suggest that Theorems \ref{Theorem}, \ref{Theorem2} can be extended to Maass forms with nontrivial level, with polynomial dependence of the implied constant on the respective levels of the forms.
    \end{remark}

     \paragraph*{\bf{Notation}} Throughout the paper, $\eps$ is an arbitrary positive real number, all of them may be different at each occurrence. By $e(x)$ we denote the exponential $e^{2\pi i x}$. For $y>0$ the notation $x\ll_{\alpha_1,\alpha_2,..\alpha_k} y$ will mean that $|x|\leq Cy$ for some constant $C>0$ depending on the parameters $\alpha_1,\alpha_2,...,\alpha_k$. In most applications we will ignore dependencies on the parameters $\alpha_i$ and just write $x\ll y$. By $x\asymp y$ we mean $C_1y\leq |x|\leq C_2y$ and $q\sim C$ means $C< q\leq 2C$.

	\section{Preliminaries}\label{prelim}
	\subsection{Maass forms on $\text{SL}_2(\mathbb{Z})$} We recall basic definitions and notions concerning Maass cusp forms on the full modular group $\text{SL}_2(\mathbb{Z})$. For $t_j\in i\mathbb{R}\cup [-1/2,1/2]$, let $H_{t_j}(1)$ denote the space of (weight zero) Hecke Maass cusp forms of spectral parameter $t_j$. Every $f\in H_{t_j}(1)$ is an eigenfunction of all the Hecke operators and admit a Fourier decomposition
	\begin{align}
		f(z)=y^{1/2}\sum_{n\neq 0}\lambda_f(n)n^{-1/2}K_{it_j}(2\pi|n|y)e(nx).
	\end{align}
	If $f$ is normalized so that $\lambda_f(1)=1,$ we call it a normalized Hecke-Maass cusp form. The Fourier coefficients of a normalized Hecke-Maass cusp form are equal to the corresponding Hecke eigenvalues, which in particular implies that they are real numbers.  From Rankin Selberg theory the following Ramanujan Bound on average ,
	\begin{align}\label{2.2}
		\sum_{n\leq N} |\lambda_f(n)|^2\ll_\varepsilon(1+|t_j|)^\varepsilon N^{1+\varepsilon},
	\end{align}
    has been established for the Fourier coefficients of a normalized Hecke form.
	In the same context, if $\theta_1$ is the best known exponent towards the generalized Ramanujan conjecture, i.e
    \begin{align}
        \lambda_f(n)\ll_\eps n^{\theta_1+\eps}, 
    \end{align}
    then $\theta_1\leq 7/64$; which is due to Kim-Sarnak\cite{KS03}.
	To such a normalized form $f$ we can attach an $L$-function given by
	\begin{align}
		L(s,f)= \sum_{n\geq 1}\frac{\lambda_f(n)}{n^{s}}= \prod_{p<\infty}\left(1-\frac{\alpha_{f,1}(p)}{p^s}\right)^{-1}\left(1-\frac{\alpha_{f,2}(p)}{p^s}\right)^{-1},
	\end{align}
	where 
	\[\alpha_f(p)+\beta_f(p)=\lambda_f(p),\;\; \alpha_f(p)\beta_f(p)=1.\]
	The $L$-function $L(s,f)$ has an analytic continuation to the whole complex plane.
	Multiplying the product of Gamma factors
	\begin{align}
		\gamma(s,f)= \pi^{-s}\Gamma\left(\frac{s+\epsilon+it_j}{2}\right)\Gamma\left(\frac{s+\epsilon-it_j}{2}\right),
	\end{align}
	with $\epsilon=0$ if $f$ is even and $\epsilon=1$ if $f$ is odd, we form the completed $L$-function
	\begin{align}
		\Lambda(s,f)=\gamma(s,f)L(s,f),
	\end{align}
	which satisfies the functional equation 
	\[\Lambda(1-s,f)=\varepsilon(f)\Lambda(s,\overline f),\]
	where $\varepsilon(f)$ is the root number with $|\varepsilon(f)|=1$ and $\bar{f}$, the dual form.
    
	Next we illustrate a transformation formula which captures the automorphy of the form $f$ in terms of weighted sums of its Fourier coefficients. Such formulae are collectively known as a Voronoi type formula.
	\begin{proposition}[Voronoi Summation Formula]\label{Voronoi_Maass}
		Let $f$ be a normalized Hecke Maass cusp form as above, $\lambda_f(n)$ be its Fourier coefficients and $g$ be a compactly supported, smooth function on $(0,\infty)$. Let $a,q\in \mathbb{Z}$ with $(a,q)=1.$ Then
		\begin{align}\label{2.7}
			\sum_{n\geq 1} \lambda_f(n)e\left(\frac{an}{q}\right)g(n)=q\sum_{\pm} \sum_{n\geq 1}\frac{\lambda_f(n)}{n} e\left(\mp \frac{\bar an}{q}\right)G^{\pm} \left(\frac{n}{q^2}\right),
		\end{align}
		where,
		\begin{equation}\label{2.8}
			\begin{split}
				G^\pm(y) & =  \frac{\epsilon_f^{(1\mp 1)/2}}{4\pi^2 i}  \int_{(\sigma)} (\pi^2 y)^{-s} \left( \frac{\Gamma(\frac{1+s+it_f}{2})\Gamma(\frac{1+s-it_f}{2})} {\Gamma(\frac{-s+it_f}{2})\Gamma(\frac{-s-it_f}{2})} \mp \frac{\Gamma(\frac{2+s+it_f}{2})\Gamma(\frac{2+s-it_f}{2})} {\Gamma(\frac{1-s+it_f}{2})\Gamma(\frac{1-s-it_f}{2})} \right) \tilde{g}(-s)  ds \\
				& = \epsilon_f^{(1\mp 1)/2} y
				\int_{0}^{\infty} g(z) J^\pm_{f}\left(  4\pi\sqrt{yz} \right) dz,
			\end{split}
		\end{equation}
		with $\sigma>\theta_1-1$  and $\tilde{g}(s) = \int_{0}^{\infty} g(z) z^{s-1}  dx$ the Mellin transform of $g$, and
		\[
		J^+_{f}\left( z \right) = \frac{-\pi}{\sin(\pi i t_f)} \left( J_{2it_f}(z) - J_{-2it_f}(z) \right), \qquad
		J^-_f(z) = 4\cosh(\pi t_f) K_{2it_f}(z).
		\]
	\end{proposition}
        \begin{proof}
            See \cite{MS06}  Equations (1.12) and (1.15) and \cite{KMV02} Appendix A.
        \end{proof}

        Finally along with the Ramanujan bound on average \eqref{2.2}, we also have the following estimate on the Fourier coefficients in the $L^4$-sense, 
        \begin{Lemma}\label{L^4_norm_Fourier_Coeff}
        \begin{align*}
            \sum_{n\leq N} \frac{|\lambda_f(n)|^4}{n}\ll_\eps (1+|t_f|)^\eps N^\eps.
        \end{align*}
        \end{Lemma}
        \begin{proof}
            See \cite{HL94} Lemma 2.1.
        \end{proof}
	\subsection{Rankin Selberg \texorpdfstring{$L$-functions}{L-functions}} Given normalized Hecke Maass cusp forms $f$ and $g$ with spectral parameters $t_f$ and $t_g$, the Rankin Selberg convolution $L$-function is defined as 
	\begin{align}
		L(s,f\otimes g)&=\zeta(2s)\sum_{n\geq 1}\frac{\lambda_f(n)\lambda_g(n)}{n^{s}}= \prod_{i=1}^2\prod_{j=1}^2\left(1-\frac{\alpha_{f,i}(p)\alpha_{g,j}(p)}{p^s}\right)^{-1},
	\end{align}
	a degree $4$ Euler product. Rankin and Selberg proved that $L(s,f\otimes g)$ admits analytic continuation to the whole complex plane except when $f=\overline{g},$ in which case there are simple poles at $s=0,1$. Moreover the completed $L$-function satisfies the functional equation 
	\[\Lambda(s,f\otimes g)=\epsilon(f\otimes g)\Lambda(1-s,\overline f\otimes \overline g),\]
	where \[\Lambda(s,f\otimes g)= \gamma(s,f\otimes g)L(s,f\otimes g),\]
	$|\epsilon(f\otimes g)|=1.$ The Gamma factor $\gamma(s,f\otimes g),$ has the following expression
	\begin{align}\label{2.10}
		\gamma(s,f\otimes g)=\pi^{-2s} &\Gamma\left(\frac{s+i(t_f+t_g)+\nu}{2}\right)\Gamma\left(\frac{s-i(t_f+t_g)+\nu}{2}\right)\\ \nonumber
		&\times\Gamma\left(\frac{s+i(t_f-t_g)+\nu}{2}\right)\Gamma\left(\frac{s-i(t_f-t_g)+\nu}{2}\right),
	\end{align}
	with $\nu=0$ if $f$ and $g$ have the same parity and $\nu=1$ otherwise. 
	\subsection{The delta symbol} Let $\delta :\mathbb{Z}\rightarrow \{0,1\}$ be the Kronecker delta function defined as 
	\begin{align}
		\delta(n) = \begin{cases}
			1 & \text{if }\;\; n=0\\
			0 & \text{otherwise}
		\end{cases}
	\end{align}
	We will use $\delta(n)$ to separate oscillations in sums of the form
	\begin{align}
		\sum_{n\sim N} a(n)b(n),
	\end{align}
	where $\{a(n)\}_{n\geq 1}, \{b(n)\}_{n\geq 1}$ are arithmetic sequences of interest. We seek a suitable expansion of $\delta(n)$ in terms of trigonometric polynomials. The one we mention and will be using is the expansion due to Duke, Friedlander and Iwaniec. For this we pick any $Q\geq 2$. Then
	\begin{equation}\label{2.14}
		\delta(n)= \frac{1}{Q}\sum_{1\leq q\leq Q} \frac{1}{q}\; \sideset{}{^*}\sum_{a\bmod q} e\left(\frac{an}{q}\right)\int_{\mathbb{R}}g(q,x)e\left(\frac{nx}{qQ}\right)dx,
	\end{equation}
	where $g(q,x)$ is an analytic weight function satisfying 
	\begin{align}\label{properties_g(q,x)}
		g(q,x) &=1+h(q,x),\;\;\;\text{with}\;\;\;h(q,x)=O\left(\frac{1}{qQ}\left(\frac{q}{Q}+|x|\right)^A\right),\\
		\nonumber g(q,x)&\ll |x|^{-A}, \; \; \text{for any}\; A>1,\\ \nonumber
		x^j\frac{\partial^j}{\partial x^j} &g(q,x) \ll \log Q\min\left\{\frac{Q}{q},\frac{1}{|x|}\right\},
	\end{align}
	for any $j\geq 1$.
	The second property implies that the effective range of the integral in \eqref{2.14} is $[-Q^{\varepsilon},Q^\varepsilon]$. Also if $q\ll Q^{1-\varepsilon}$ and $x\ll Q^{-\varepsilon}$, then $g(q,x)$ can be replaced by $1$ at the cost of a negligible error. In the complimentary range we have $x^j \frac{\partial^j}{\partial x^j}g(q,x)\ll_{j,\eps} Q^{(j+1)\varepsilon}$, for any $j\geq 1$. Finally by Parseval and Cauchy we get 
	$$\int (|g(q,x)|+|g(q,x)|^2)dx\ll Q^\varepsilon.$$
	We summarize these observations in the following lemma 
	\begin{Lemma}\label{DFI_delta}
		For $Q\geq 2,$ one has 
		\begin{align}
			\delta(n)= \frac{1}{Q}\sum_{1\leq q\leq Q} \frac{1}{q}\; \sideset{}{^*}\sum_{a\bmod q} e\left(\frac{an}{q}\right)\int_{\mathbb{R}}W\left(\frac{x}{Q^\varepsilon}\right)g(q,x)e\left(\frac{nx}{qQ}\right)dx + O_A(Q^{-A}),
		\end{align}
		where $W$ is a nonnegative smooth function supported on $[-2,2]$ with $W\equiv 1$ on $[-1,1]$, $W^{(j)}\ll_j 1$ for all $j\geq 0$ and $g(q,x)$ satisfies the properties in \eqref{properties_g(q,x)}.
	\end{Lemma}
    \begin{proof}
        See \cite{IK04}, Chapter $20$ and \cite{Hu21}, Lemma $15$.
    \end{proof}
	\subsection{Bessel functions} We quote \S\S $2.8$ of \cite{Hu24} mentioning a few properties of Bessel functions which we will need later on, while analyzing integral transforms arising from the Voronoi summation formula.  Let $r>0$ and $x\gg T^\varepsilon r,$ we have 
	\begin{equation}\label{2.16}
		\frac{J_{\pm 2ir}(2x)}{\sin \pi i r}= \frac{W(x)}{(4r^2+x^2)^{1/4}}  \exp{\left(\pm 2i\omega(x,r)\right)} + O_A(x^{-A})
	\end{equation}
	with
	\begin{equation*}
		\omega(x,r) = \left(r^2+x^2\right)^{1/2}\mp r\sinh^{-1}(r/x)
	\end{equation*}
	and 
	\begin{equation}\label{2.17}
		K_{2i\tau}(2x) \cosh(\pi \tau)  \ll x^{-1/2} \exp(-2x+ \pi|\tau|) \ll x^{-6} \exp(-x).
	\end{equation}
	\subsection{Stirling Approximation for Gamma functions} This section is borrowed from \S\S $2.7$ of \cite{Hu24}.
    
	For fixed $\sigma\in\mathbb{R}$, real $|t|\geq1000$ and any $A>0$, we have Stirling's formula
	\begin{equation*}
		\Gamma(\sigma+it) = e^{-\frac{\pi}{2}|t|} |t|^{\sigma-\frac{1}{2}} \exp\left( it\log\frac{|t|}{e} \right) \left( g_{\sigma,A}(t) + O_{\sigma,J}(|t|^{-A}) \right),
	\end{equation*}
	where
	\[
	t^j \frac{\partial^j}{\partial t^j} g_{\sigma,A}(t) \ll_{j,\sigma,A} 1
	\]
	for all fixed $j\in \mathbb{N}_0$.
	Similarly, we have
	\begin{equation*}
		\frac{1}{\Gamma(\sigma+it)} = e^{\frac{\pi}{2}|t|} |t|^{-\sigma+\frac{1}{2}} \exp\left( -it\log\frac{|t|}{e} \right) \left(  h_{\sigma,A}(t) + O_{\sigma,A}(|t|^{-A}) \right),
	\end{equation*}
	where
	\[
	t^j \frac{\partial^j}{\partial t^j} h_{\sigma,A}(t) \ll_{j,\sigma,A} 1
	\]
	for all fixed $j\in \mathbb{N}_0$.
	Hence
	\begin{equation}\label{2.18}
		\frac{\Gamma(\sigma+it)}{\Gamma(\sigma-it)}
		= \exp\left( 2it\log\frac{|t|}{e} \right)
		\left( w_{\sigma,A}(t) + O_{\sigma,A}(|t|^{-A}) \right),
	\end{equation}
	where
	\[
	t^j \frac{\partial^j}{\partial t^j} w_{\sigma,A}(t) \ll_{j,\sigma,A} 1
	\]
	for all fixed $j\in \mathbb{N}_0$.
	\subsection{Oscillatory integrals}

	We borrow the following variants of integration by parts and stationary phase Lemmas from \S\S $8$ of \cite{BKY13} .

	\begin{Lemma}\label{repeated_integration_by_parts}
		Let $Y\geq1$. Let $X_0,\; V_0,\; R_0,\; Q_0>0$ and suppose that $w$ is a smooth function with  $ \mathrm{supp} w \subseteq [\alpha,\beta]$ satisfying $w^{(j)}(\xi) \ll_j X_0 V_0^{-j}$ for all $j\geq0$.
		Suppose that on the support of $w$, $h$ is smooth and satisfies that
		$h'(\xi)\gg R_0$ and $ h^{(j)}(\xi) \ll Y_0 Q_0^{-j}$, for all $j\geq2.$
		Then for arbitrarily large $A$ we have
		\[
		I = \int_{\mathbb{R}} w(\xi) e(h(\xi))  d \xi  \ll_A (\beta-\alpha)  X_0 \left[  \left(\frac{Q_0R_0}{\sqrt{Y_0}}\right)^{-A} + (R_0V_0)^{-A}  \right].
		\]
	\end{Lemma}
        \begin{proof}
            See \cite{BKY13} Lemma $8.1$
        \end{proof} 
	\begin{proposition}\label{stationary_phase}
		Let $0 < \delta < 1/10$, $X_0, Y_0, V_0, \Omega, Q_0 > 0$, $Z := Q_0 + X_0 + Y_0 + \Omega+1$,  and assume that
		\begin{equation}\label{importantconditions}
			Y_0 \geq Z^{3 \delta}, \quad \Omega \geq V_0 \geq \frac{Q_0Z^{ \frac{\delta}{2}} }{Y_0^{1/2}}.\end{equation} Suppose that $w$ is a smooth function on $\Bbb{R}$ with support on an  interval $J$ of length $\Omega$, satisfying
		\begin{equation*}
			w^{(j)}(t) \ll_j X_0 V_0^{-j}
		\end{equation*}
		for all $j \in \Bbb{N}_0$. Suppose $h$ is a smooth function on $J$ such that there exists a unique point $t_0 \in J$ such that $h'(t_0) = 0$, and furthermore
		\begin{equation}\label{diffh}
			h''(t) \gg Y_0 Q_0^{-2}, \quad h^{(j)}(t) \ll_j Y_0 Q_0^{-j}, \qquad \text{for } j=1,2, 3, \dots, t \in J. 
		\end{equation}
		Then the integral $I  $ in Lemma \ref{repeated_integration_by_parts}
		has an asymptotic expansion of the form
		\begin{equation}
			\label{eq:statphase}
			I  = \frac{ e^{ih(t_0)}}{\sqrt{h''(t_0)}} \sum_{n \leq 3 \delta^{-1} A} p_n(t_0)  + O_{A,\delta}(Z^{-A}), \quad p_n(t_0) = \frac{\sqrt{2\pi} e^{\pi i/4}}{n!} \Big(\frac{i}{2 h''(t_0)}\Big)^n  G^{(2n)}(t_0),
		\end{equation}
		where $A > 0$ is arbitrary, and
		\begin{equation}\label{defG}
			G(t) = w(t) e^{i H(t)}, \qquad H(t) = h(t) - h(t_0) - \frac{1}{2} h''(t_0) (t-t_0)^2.
		\end{equation}
		Furthermore, each $p_n$ is a rational function in $h'', h''', \dots$, satisfying
		\begin{equation}
			\label{eq:pnderiv}
			\frac{d^j}{d t_0^j} p_n(t_0) \ll_{j,n} X_0(V_0^{-j} + Q_0^{-j}) \big((V_0^2 Y_0/Q_0^2)^{-n} + Y_0^{-n/3}\big).
		\end{equation}
	\end{proposition}
        \begin{proof}
            See \cite{BKY13} Proposition $8.2$.
        \end{proof}
	\subsection{Weight Functions} We mention some conventions about smooth weight functions from \S\S $2.5$ of \cite{Hu24}. Let $F$ be an index set and $X=X_i :F\rightarrow \mathbb{R}_{\geq1}$ be a function of $i \in F$. A family of $\{w_i\}_{i\in F}$ of smooth functions supported on a product of dyadic intervals in $\mathbb{R}_{>0}^d$ is called \emph{$X$-inert} if for each $j=(j_1,\ldots,j_d) \in \mathbb{Z}_{\geq0}^d$ we have
	\[
	\sup_{i\in{F}} \sup_{(x_1,\ldots,x_d) \in \mathbb{R}_{>0}^d}
	X_i^{-j_1-\cdots -j_d} \left| x_1^{j_1} \cdots x_d^{j_d} w_T^{(j_1,\ldots,j_d)} (x_1,\ldots,x_d) \right|
	\ll_{j_1,\ldots,j_d} 1.
	\]

	For a $T^\varepsilon$-inert function $V$, we may separate variables in $V(x_1, \ldots , x_d)$ by first inserting a redundant function $V (x_1) \cdots V (x_d)$ that is 1 on the support of $V$ and then applying Mellin inversion
	\begin{multline*}
		V (x_1, \ldots , x_d) = V (x_1, \ldots , x_d)V (x_1) \cdots V (x_d)
		\\
		=  \frac{1}{(2\pi i)^d} \int_{(\varepsilon)}\cdots \int_{(\varepsilon)} \tilde{V}(s_1,\ldots,s_d)
		(V (x_1) \cdots V (x_d) x_1^{-s_1} \cdots x_n^{-s_d} ) d s_1 \cdots d s_d,
	\end{multline*}
	where $\tilde{V}(s_1,\ldots,s_d)=\int_{0}^{\infty}\cdots\int_{0}^{\infty} V(x_1, \ldots , x_d) x_1^{s_1-1} \cdots x_d^{s_d-1} d x_1 \cdots d x_d$ is the Mellin transform of $V$.
	Here we can truncate the vertical integrals at height $|\Im s_j| \ll T^{2\varepsilon}$ at the cost of a negligible error $O_A(T^{-A})$.
	We will often separate variables in this way without explicit mention.
	\subsection{Initial setup and outline of the proof}
	Let $f$ and $g$ be defined as earlier, with spectral parameters $t_f, t_g\asymp T$ respctively, with $T\rightarrow \infty$. Let $\lambda_f(n)$ and $\lambda_g(n)$ be the normalized Fourier coefficients of $f$ and $g$ respectively. We want to analyse the Rankin Selberg $L$-series $L(s,f\otimes g)$ at the point $s=1/2+it$, where $t=t_f+t_g$. Our initial strategy is to  to express $L(1/2+it,f\otimes g)$ as a weighted Dirichlet Series, which corresponds to taking a smooth dyadic partition of its approximate functional equation.
	\begin{Lemma}\label{Approx_func_eq}
		Let $0<\theta<3/4$, and $T\rightarrow \infty$. With $t=t_f+t_g$, $t_f,t_g\asymp T$, one has 
		\begin{align}\label{2.24}
			L(1/2+it,f\otimes g)\ll  T^\varepsilon\sup_{T^{3/2-\theta}\ll N\ll  T^{3/2+\varepsilon}} \frac{|S(N)|}{\sqrt{N}}+ T^{3/4-\theta+\varepsilon},
		\end{align}
		where
		\begin{align}
			S(N)= \sum_{n=1}^{\infty}\lambda_f(n)\lambda_g(n)n^{-it}V(n/N),
		\end{align}
		for some smooth function supported in $[1,2]$, satisfying $V^{(j)}\ll_j 1$ for all $j\geq 0$ and normalised so that $\int V(y)\mathop{dy}=1.$
	\end{Lemma}
	\begin{proof} 
		See \cite{IK04}, Section $5.2$. Using Cauchy's inequality and Ramanujan bound on average \eqref{2.2} for the range $N\ll T^{3/2-\theta},$ one obtains the second term in the RHS of \eqref{2.24}.
	\end{proof}
	\begin{remark} 
		Upon estimating $S(N)$ using Cauchy's inequality and Ramanujan bound on average \eqref{2.2}, we get $L(1/2+it,f\otimes g)\ll_\varepsilon T^{3/4+\varepsilon}$. Hence, in order to obtain subconvexity we need to obtain cancellations in $S(N),$ for $N$ roughly of the size $T^{3/4}$.
	\end{remark}
	\subsubsection{Application of the delta symbol} Writing $t=t_f+t_g$, we follow \cite{Mun22} to separate the oscillatory terms $\lambda_f(n)n^{-it_f}$ and and $\lambda_g(n)n^{-it_g}$ involved in $S(N)$. Expressing $S(N)$ as
	\begin{align}
		\sum_{n,m=1}^\infty \lambda_f(n)n^{-it_f}\lambda_g(m)m^{-it_g}\delta(n-m),
	\end{align}
	we use Lemma \ref{DFI_delta} with $Q=\sqrt{N/K}$ for some $T^\varepsilon\ll K=T^{1-\eta}\ll T^{1-\varepsilon}$ ($\eta>0$ to be chosen optimally later) to expand $\delta(n-m)$ and obtain that
	\begin{align}\label{2.27}
		S(N)= & \frac{1}{Q}\int_{\mathbb{R}}W\left(\frac{x}{Q^\varepsilon}\right)\sum_{1\leq q\leq Q} \frac{g(q,x)}{q}\; \sideset{}{^*}\sum_{a\bmod q}\\
		\nonumber &\quad \quad \sum_{n=1}^\infty \lambda_f(n)e\left(\frac{an}{q}\right)n^{-it_f}e\left(\frac{nx}{qQ}\right)V\left(\frac{n}{N}\right)\\
		\nonumber &\quad\;\;\;\, \sum_{m=1}^\infty \lambda_g(m) e\left(-\frac{am}{q}\right)m^{-it_g}e\left(-\frac{mx}{qQ}\right)U\left(\frac{m}{N}\right)dx+ O_A(T^{-A}),
	\end{align}
	upto a negligible error term, $U$ being a non-negative smooth function supported on $[1/2,5/2]$, with $U\equiv 1$ on $[1,2]$ and satisfying $U^{(j)}\ll_j 1$ for all $j\geq 0$. Notice that here we do not incorporate the $v$-integral, devised by Munshi (\cite{Mun15}, \cite{Mun22}), as the conductor lowering mechanism. Choosing $Q=\sqrt{N/K}$, inherently introduces an extra oscillation from (of size comparable to $K$ generically) the $x$- integral to the delta symbol. This is not harmful per se as we will be able to remove the $x$-integral completely by repeated integration by parts and the first derivative estimate (see \S \ref{simplification_of_I}).
	
	At this point if we estimate $S(N)$ trivially, employing Cauchy's inequality and Ramanujan bound on average \eqref{2.2}, we obtain $S(N)\ll_\varepsilon N^{2+\varepsilon}$. Therefore for cancellation we need to save $N$ plus something extra in $S(N)$. Before moving forward we introduce a smooth dyadic partition of unity in the $x$-integral and the $q$-sum. Taking supremum over the dydadic sums we have
	\begin{align}\label{2.28}
		S(N) \ll_{\varepsilon} T^\varepsilon\sup_{\substack{1\ll C\ll Q \\ T^{-100} \ll X\ll Q^\varepsilon\\ \pm}} |S_{\pm}(N,C,X)| + T^{-96+\varepsilon},
	\end{align}
	where 
	\begin{align}\label{2.29}
		S_{\pm}(N,C,X) &= \frac{1}{Q}\int_{\mathbb{R}}W\left(\frac{\pm x}{X} \right)\sum_{q\sim C} \frac{g(q,x)}{q}\; \sideset{}{^*}\sum_{a\bmod q}\\
        \nonumber
		&\quad \quad \sum_{n=1}^\infty \lambda_f(n)e\left(\frac{an}{q}\right)n^{-it_f}e\left(\frac{nx}{qQ}\right)V\left(\frac{n}{N}\right)\\
        \nonumber
		&\quad\;\;\;\, \sum_{m=1}^\infty \lambda_g(m) e\left(-\frac{am}{q}\right)m^{-it_g}e\left(-\frac{mx}{qQ}\right)U\left(\frac{m}{N}\right)dx.
	\end{align}
	\begin{remark}
		The rest of the paper is dedicated towards obtaining estimates $S_+(N,C,X)$ only. The estimation of $S_-(N,C,X)$ would be exactly similar after an initial change of variables $x\rightsquigarrow -x.$
	\end{remark}
        \begin{remark}
            The merit of keeping the $GL(1)$ oscillation $n^{it_f}\,(\,/\, n^{it_g})$ together with the respective Hecke-eigenvalue (or Fourier coefficient) $\lambda_f(n)\,(\,/\lambda_g(m))$ is that individual conductors of the $n$, $m$ sums decrease from $T^2Q^2$ to $TKQ^2$ (in the generic case), enabling us to save significantly from the Voronoi summation formula. 
        \end{remark}
	\subsubsection{Sketch of the proof}\label{sketch}
	We now discuss a brief sketch of our proof. For the sketch we assume generically that $N=T^{3/2}$ (then $Q=\sqrt{N/K}=T^{3/4}/\sqrt{K}$), $X= 1$, $C=Q$  and $K>T^{1/2}$ and denote by abuse of notation $S_+(N,C,X)$ as $S(N)$. The reason behind the assumption on $K$ will be apparent as the details of the proof evolve in the later sections. Let $t_{f,g}=\frac{t_f}{t_g}\asymp 1$.

    As a standard procedure after separating oscillations, we apply Voronoi summation formula to the $n$ and $m$ sums, which then turn out to be
    \begin{align}
        q\sum_{n\geq 1}\frac{\lambda_f(n)}{n}e\left(-\frac{\bar{a}n}{q}\right)G^1_x\left(\frac{n}{q^2}\right)
    \end{align}
    and
    \begin{align}
        q\sum_{m\geq 1}\frac{\lambda_g(n)}{m}e\left(\frac{\bar{a}m}{q}\right)G^2_x\left(\frac{m}{q^2}\right).
    \end{align}
    The analysis of the integral transforms $G^j_x(y)$ ($j=1,2$) is performed in \S \ref{Prem_Int_transform}, where (in the generic case) we obtain upto bounded scalar multiples and negligible error terms,
    \begin{align}
        G^1_x\left(\frac{n}{q^2}\right)=N^{1/2}K^{1/2}&\left(\frac{n}{q^2}\right)^{1/2+it_f} \\
            \nonumber
            &\int V_1(\tau_1)e\left(\frac{1}{2\pi} \left(K\tau_1 \log\frac{xq}{2\pi^2nQ}+(K\tau_1-2t_f)\log\frac{2t_f-K\tau_1}{2e}\right)\right)d\tau
    \end{align}
    and
    \begin{align}
        G^2_x\left(\frac{m}{q^2}\right)=N^{1/2}&K^{1/2}\left(\frac{m}{q^2}\right)^{1/2+it_g}\\
            \nonumber
            &\int U_2(\tau_2)e\left(\frac{1}{2\pi} \left(-K\tau_2 \log\frac{xq}{2\pi^2mQ}-(K\tau_2+2t_g)\log\frac{2t_g+K\tau_2}{2e}\right)\right)d\tau_2,
    \end{align}
    for compactly supported $T^\eps$-inert weight functions $V_1$ and $U_2$ and 
    with contributions only due to the dual range $m,n\asymp KTQ^2/N=T$. Thus 
    \begin{align}
        S(N)\ll \frac{NK}{Q}&\sum_{q\sim Q}\frac{1}{q^{1+2it}} \sideset{}{^*}\sum_{a\bmod q}\\ \nonumber&\sum_{\substack{n\asymp T\\m\asymp T}} \frac{\lambda_f(n)\lambda_g(m)}{n^{1/2-it_f}m^{1/2-it_g}} e\left(\frac{\bar{a}(m-n)}{q}\right)\mathcal{I}(m,n,q),
    \end{align}
    where $\mathcal{I}(m,n,q)$ is the three fold integral
    \begin{align*}
        \int_{x\asymp 1}g(q,x)\int\int V_1(\tau_1)U_2(\tau_2)e\left(\frac{1}{2\pi}h(\tau_1,\tau_2,x)\right)\mathop{dx}.
    \end{align*}
    where 
    \begin{align*}
        h(\tau_1,\tau_2,x)&= K\tau_1 \log\frac{xq}{2\pi^2nQ}+(K\tau_1-2t_f)\log\frac{2t_f-K\tau_1}{2e}\\
        &-K\tau_2\log\frac{xq}{2\pi^2mQ}-(K\tau_2+2t_g)\log\frac{2t_g+K\tau_2}{2e}.
    \end{align*}
    Next we simplify $\mathcal{I}(m,n,q)$ using stationary phase analysis and obtain further restrictions on the dual ranges. This part forms the one of the key ingredients in the proof as we are able to save considerably more than the usual (square root) from the integrals. The idea is to first exploit the $x$-integral using integration by parts to obtain the restriction $\tau_1-\tau_2\ll 1/K$. Changing variables $\tau_2=\tau_1+u$, with $|u|\ll 1/K$ and estimating the $\tau_1$-integral by stationary phase analysis we obtain that when $n-mt_{f,g}\asymp K$ it approximately equals,
    \begin{align*}
        \frac{\sqrt{T}}{K}V_3\left(\frac{n-mt_{f,g}}{K}\right)e\left(\frac{1}{2\pi}\left(2t\log \frac{e(n+m)}{t_f+t_g}-2t_f\log n-2t_g\log m\right)\right),
    \end{align*}
    for a compactly supported and $T^\eps$-inert function $V_3$
    and is negligibly small otherwise. The $a$-sum modulo $q$, which is a Ramanujan sum, generically reduces to the congruence condition
    \begin{align}
        q\mathds{1}_{m\equiv n\bmod q}.
    \end{align}
    Then, using Ramanujan bound on average for the Fourier coefficients, $S(N)$ can be trivially bounded as,
    \begin{align}
        S(N)\ll T^\eps \frac{NK}{Q}\times \frac{1}{K}\times \frac{\sqrt{T}}{K}\times \frac{1}{T}\times TK= T^\eps \frac{N\sqrt{T}}{Q}.
    \end{align}
    Thus we are left with saving $\sqrt T/Q$ plus some more in the sum
    \begin{align}
        S'(N)= &\sum_{q\sim Q}\frac{1}{q^{2it}} \\ \nonumber& \sum_{\substack{n,m\asymp T\\ m\equiv n\bmod q}}\
         \frac{\lambda_f(n)\lambda_g(m)}{n^{1/2+it_f}m^{1/2+it_g}} V_3\left(\frac{n-mt_{f,g}}{K}\right)e\left(\frac{t\log(n+m)}{\pi}\right)   ,
    \end{align}
    in order to reach subconvexity. Note that the support condition on $n-mt_{f,g}$ is captured by the weight $V_3$. Now the strategy forward is to break the involution by Cauchy's inequality on one of the variables $n$ or $m$ and use Poisson summation formula. Taking absolute value inside the $q$ and $n$-sums, we apply Cauchy's inequality on $n$ followed by Ramanujan bound on average and obtain
    \begin{align}\label{2.38}
        S'(N)\ll & T^\eps \sum_{q\sim Q}\\ \nonumber & \left(\sum_{n\asymp T}\left|\sum_{\substack{m\asymp T\\ m\equiv n\bmod q}}\frac{\lambda_g(m)}{m^{1/2+it_g}} V_3\left(\frac{n-mt_{f,g}}{K}\right)e\left(\frac{t\log(n+m)}{\pi}\right)\right|^2\right)^{1/2}.
    \end{align}
    Notice that the required saving from the sum inside the square root is now at least $T/Q^2$ and it is trivially bounded  by $T^\eps(K/Q)^2$. Opening up the absolute value square and interchanging summations, the quantity inside the square root equals
    \begin{align}
        &\sum_{\substack{m_1,m_2\asymp T\\ m_1-m_2\ll K\\m_1\equiv m_2\bmod q}}\frac{\lambda_g(m_1){\lambda_g(m_2)}}{m_1^{1/2+it_g}m_2^{1/2-it_g}}\\ \nonumber&\sum_{\substack{n\asymp T\\ n\equiv m_1\bmod q}} V_3\left(\frac{n-m_1t_{f,g}}{K}\right)\overline{V_3\left(\frac{n-m_2t_{f,g}}{K}\right)} e\left(\frac{t}{\pi}(\log(n+m_1)-\log(n+m_2))\right).
    \end{align} 
    By symmetry it is enough to only consider the contribution due to the terms $0\leq m_2-m_1\ll K$
    In the diagonal ($m_1=m_2$) we save $K/Q$, which is enough if $K/Q>T/Q^2$ i.e $$K^{1/2}>T^{1/4}\iff K>T^{1/2}.$$
     In the off diagonal ($m_1\neq m_2$) smoothing out the $n$-sum by an appropriate weight function $\varphi$ and writing $n=rq+m_1$, we apply Poisson summation on $r$. Hence the $n$-sum after Poisson reduces to
    \begin{align}
        \frac{1}{2\pi i} \sum_{r} \int& \varphi\left(\frac{wq+m_1}{T}\right)V_3\left(\frac{wq+m_1-m_1t_{f,g}}{K}\right)\overline{V_3\left(\frac{wq+m_1-m_2t_{f,g}}{K}\right)}\\ \nonumber& e\left(\frac{t}{\pi}(\log(wq+2m_1)-\log(wq+m_1+m_2))-rw\right)\mathop{dw}.
    \end{align}
    With the change of variables $(wq+m_1-m_1t_{f,g})/K\rightsquigarrow w$, the Fourier transform becomes
    \begin{align}\label{2.41}
        \frac{K}{q}&e\left(\frac{-rm_1(t_{f,g}-1)}{q}\right)\\ \nonumber &\int\Phi(w) e\left(\frac{t}{\pi}(\log(Kw+m_1t_{f,g}+m_1))-\log(Kw+m_1t_{f,g}+m_2))-\frac{rKw}{q}\right)\mathop{dw},
    \end{align}
    where 
    \[\Phi(w)= \varphi\left(\frac{Kw+m_1t_{f,g}}{T}\right)V_3\left(w\right)\overline{V_3\left(\frac{Kw-(m_2-m_1)t_{f,g}}{K}\right)}\]
    is non oscillatory.
    The analysis of the exponential integral in \eqref{2.41} forms another key step in our analysis. Basically the idea is to expand the logarithms in the phase by Taylor series expansion and observe that if $K$ is restricted to the interval $(T^{1/2},T^{2/3})$, then the phase function essentially becomes linear and equals
    \begin{align}
        \frac{tK(m_2-m_1)w}{\pi(m_1+m_1t_{f,g})(m_2+m_1t_{f,g})}-\frac{rKw}{q}+\tilde{h}(w),
    \end{align}
    such that $\tilde h^{(j)}(w)\ll_j T^\eps.$ Then by repeated integration by parts, $r$ gets restricted to the dual range 
    \begin{align}\label{2.42}
        \frac{tQ(m_2-m_1)}{\pi (m_1+m_1t_{f,g})(m_2+m_1t_{f,g})} -r\ll \frac{T^\eps Q}{K}.
    \end{align}
     In the zero frequency ($r=0$) we obtain the restriction $m_2-m_1\ll T/K$, where we save $K^2/T$ and it is enough if 
    $$K^2/T>T/{Q^2}\implies K>T^{2}/T^{3/2}=T^{1/2}.$$
    In the non zero frequency ($r\neq 0$), a trivial estimate does not save us anything. Instead, observing that the dual length of $r$  is $0<r\ll KQ/T$ and writing $m_2=m_1+qh$ with $0<h\ll K/Q$, we determine precisely all such $m_1$ which satisfy the condition in \eqref{2.42} for fixed $r$ and $h$. The steps are made precise in \S\S \ref{upto_two_third} Eqn.\eqref{7.18}-\eqref{7.20} and we end up with the admissible range 
    \begin{align}
        m_1-\alpha(r,h,q)\ll \frac{TQ}{Kr}, 
    \end{align}
    where
    \[\alpha(r,h,q)= \frac{-\pi rdh+ \sqrt{\pi^2r^2d^2h^2+4\pi td^2hr}}{2\pi r}.\]
    Altogether, the contribution due to $m_1\neq m_2$ and $r\neq 0$ is bounded by 
    \begin{align}\label{2.45}
        \frac{K}{Q}\sum_{0<r\ll KQ/T}\sum_{0< h\ll K/Q} \sum_{\substack{m\asymp T\\ m-\alpha(h,r,q)\ll \frac{TQ}{Kr}}} \frac{|\lambda_g(m)\lambda_g(m+qh)|}{m^{1/2}(m+qh)^{1/2}}.
    \end{align}
    The endgame is estimation of \eqref{2.45}. It involves going to the $L^4$-norm of the Fourier coefficients by applying A.M-G.M inequality followed by Cauchy's inequality, using Lemma \ref{L^4_norm_Fourier_Coeff}, and a point counting argument to obtain an estimate of same strength as implied by the Ramanujan conjecture. We end up with obtaining that \eqref{2.45} is bounded by $T^\eps K/Q$, saving us $K/Q$ again. The steps are made precise in Lemma \ref{estimate_of_S_1}.
    Therefore if $T^{1/2}<K<T^{2/3}$, we are able to deduce subconvexity. 

     In principle however, the upper bound on the admissible range $K$ will end up being smaller due to contributions from non generic cases which we have skipped in the sketch. The optimal value of $K$ we obtain at the end is $K=T^{25/42}$.
        \section{Application of Voronoi Summation formula}\label{apply_voronoi}
	We now apply the Voronoi Summation formula (Lemma \ref{Voronoi_Maass}) to both the $n$ and $m$-sums getting 
	\begin{align}\label{3.1}
		\sum_{\substack{n\geq1}} \lambda_f(m) e\left(\frac{an}{q}\right)
		&
		n^{-it_f}e\left(\frac{nx}{qQ}\right) V\left(\frac{n}{N}\right)
		\nonumber
		\\
		& = q \sum_{\pm_1}  \sum_{n\geq1} \frac{ \lambda_{f}(n) }{n}
		e\left(\frac{\mp\bar{a} n}{q}\right) G_x^{\pm_1} \left(\frac{n}{q^2} \right),
	\end{align} and 
	\begin{align}\label{3.2}
		\sum_{\substack{n\geq1}} \lambda_g(m) e\left(-\frac{am}{q}\right)
		&
		m^{-it_g}e\left(-\frac{mx}{qQ}\right) V\left(\frac{m}{N}\right)
		\nonumber
		\\
		& = q \sum_{\pm_2}  \sum_{m\geq1} \frac{ \lambda_{f}(m) }{m}
		e\left(\frac{\pm \bar{a} m}{q}\right) G_x^{\pm_2} \left(\frac{m}{q^2} \right),
	\end{align}
	Where $G_x^{\pm_1}$ and $G_x^{\pm_2}$ are the integral transforms as described in Lemma \ref{Voronoi_Maass} with $$g^1(z)= z^{-it_f}e\left(\frac{zx}{qQ}\right)V\left(\frac{z}{N}\right), \;\;  g^2(z)= z^{-it_g}e\left(\frac{-zx}{qQ}\right)U\left(\frac{z}{N}\right).$$ 
	\subsection{Preliminary analysis of the integral transforms}\label{Prem_Int_transform}
    In this section we present a preliminary simplification of the integral transforms $G_x^{\pm_1}$ and $G_x^{\pm_2}$. Let $B=NX/CQ$. The following Lemma gives a crude estimate on the length of the dual $n$ and $m$ sums
	\begin{Lemma}
		$G_x^{\pm_{1,2}}(y)$ is negligibly small i.e. 
		$G_x^{\pm_{1,2}}(y)\ll_A T^{-A}$ unless 
		\begin{align}
			Ny\ll T^\varepsilon(T^2+B^2)
		\end{align}
	\end{Lemma}
	\begin{proof}
		We prove the lemma for $G_x^{\pm_1}(y)$. For $G_x^{\pm_2}$,the proof is exactly the same. From the second expression in \eqref{2.8}, after a change of variable $z\rightsquigarrow Nz$
		\begin{align}
			G_x^{\pm_{1}}(y)&\asymp N^{1-it_f}y\int_{0}^\infty z^{-it_f}e\left(\frac{Nzx}{qQ}\right)V\left(z\right) J_f^{\pm}(4\pi \sqrt{Nyz})\mathop{dz}\\   
		\end{align}
		Suppose that $Ny\gg T^\varepsilon(T^2+B^2)$, then $Ny\gg T^\varepsilon t_f^2,$ and from \eqref{2.17}
		\begin{align}
			K_{2it_f}\left(4\pi\sqrt{Nyz}\right)\cosh(\pi t_f) \ll (Nyz)^{-3}\exp\left(-2\pi\sqrt{Nyz}\right)\ll_A T^{-A},
		\end{align}
		as $z\gg 1$. It follows that $G_x^{-_1}(y)\ll T^{-A}$ in this case. Next by \eqref{2.16} we have 
        \begin{align}
            G_x^{+_1}(y)\asymp y N^{1-it_f}\sum_{\pm}\int_0^\infty &V_1(z)z^{-it_f}e\left(\frac{Nzx}{qQ}\right)\\ \nonumber
            &\exp\left(\pm 2i\left(\left(t_f^2+4\pi^2Nyz\right)\mp t_f\sinh^{-1}(t_f/2\pi\sqrt{Nyz})\right)\right)\mathop{dz}
        \end{align}
        for a new compactly supported smooth function $V_1$ with bounded derivatives. Using the formula
        \[\sinh^{-1}(x)= \log\left(x+\sqrt{x^2+1}\right),\] gives us
        \begin{align}\label{}
            G_x^{+_1}(y)\asymp y N^{1-it_f}\sum_{\pm}\int_0^\infty &V_1(z)e\left( h(z)\right)\mathop{dz},
        \end{align}
        where \[h(z)= \frac{Nzx}{qQ}\pm\frac{1}{\pi}\left(\left(t_f^2+4\pi^2Nyz\right)^{1/2}\mp t_f\log\left(t_f+\left(t_f^2+4\pi^2Nyz\right)^{1/2}\right)\right),\]
        Then $h ^{(j)}(z)\asymp (Ny)^{1/2},$ for all $j\geq 1.$ Then by Lemma \ref{repeated_integration_by_parts} with $R_0=Y_0=\sqrt{Ny}\gg T^{1+\varepsilon},\;\;V_0=Q_0=X_0=1,$ the Lemma follows. 
	\end{proof}
	Next depending on the size of the oscillation $B=NX/CQ,$ we can make further simplification using the first identity in \eqref{2.8}. The approach is similar to \cite{Hu24}.
    \subsubsection{The non oscillatory range} 
         Let $B=NX/CQ\ll T^\varepsilon$. For fixed $\sigma> 0$ large and $s=\sigma+i\tau$, changing variables $\tau\rightsquigarrow \tau-t_f$ \eqref{2.8} implies,
        \begin{align}\label{3.9}
            G_x^{\pm_1}(y)=\frac{\epsilon_f^{(1\mp 1)/2}}{4\pi^2 i}\int_{\mathbb{R}}(\pi^2y)^{-\sigma-i\tau+it_f}\gamma(\sigma+i\tau-it_f)^\mp\tilde{g_1}(-\sigma-i\tau+it_f)d\tau.
        \end{align}
        The Mellin transform equals
        \begin{align}\label{3.10}
            \tilde{g_1}(-\sigma-i\tau+it_f)&= \int V\left(\frac{z}{N}\right)z^{-\sigma-1-i\tau}e\left(\frac{zx}{qQ}\right)\mathop{dz}\\
            &\nonumber= N^{-\sigma}\int V(z)z^{-\sigma-1} e\left(\frac{Nzx}{qQ}-\frac{\tau \log z}{2\pi}\right)\mathop{dz}.
        \end{align}
        Using Lemma \ref{repeated_integration_by_parts} we can again conclude that \eqref{3.10} is negligibly small unless $|\tau|\ll T^{\varepsilon}.$ Now Stiring approximation implies 
        \begin{align}\label{3.11}
            \gamma^{\mp} (\sigma+i\tau-it_f)= \frac{\Gamma(\frac{1+\sigma+i\tau}{2})\Gamma(\frac{1+\sigma+i(\tau-2t_f)}{2})} {\Gamma(\frac{-\sigma-i(\tau-2t_f)}{2})\Gamma(\frac{-\sigma-i\tau}{2})} \mp \frac{\Gamma(\frac{2+\sigma+i\tau}{2})\Gamma(\frac{2+\sigma+i(\tau-2t_f)}{2})} {\Gamma(\frac{1-\sigma-i(\tau-2t_f)}{2})\Gamma(\frac{1-\sigma-i\tau}{2})}\ll_{\sigma,\varepsilon} T^{1/2}(T^{1+\varepsilon})^\sigma. 
        \end{align}
        Therefore \eqref{3.9} and \eqref{3.11} imply
        \begin{align}\label{3.12}
            G_x^{\pm_1}(y)\ll_{\sigma,\varepsilon} T^{1/2}\left(\frac{T^{1+\varepsilon}}{yN}\right)^\sigma.
        \end{align}
        Taking $\sigma$ large enough we see that $G_x^{+_1}(y)$ is negligibly small unless $yN\ll T^{1+\varepsilon}$. Inserting $\sigma=-1/2$ in \eqref{3.12} implies
        \begin{align}\label{3.12}
            G_x^{\pm_1}(y)\ll_{\varepsilon} (yN)^{1/2}T^\varepsilon.
        \end{align}
        Simplification of $G_x^{\pm_2}(y)$ in this case is exactly same and results in the same estimate as above.
        
        Next if $B=NX/CQ\gg T^\varepsilon,$ we fix $\sigma=-1/2$. Writing $s= -1/2 + i\tau$ and applying the change of variables $\tau \rightsquigarrow \tau-t_f$ we have 
        \begin{equation}\label{3.14}
            G_x^{\pm_1}(y)=\frac{\epsilon_f^{(1\mp 1)/2}}{4\pi^2 i}\int_{\mathbb{R}}(\pi^2y)^{1/2-i\tau+it_f}\gamma^{\mp}(-1/2+i\tau-it_f)\tilde{g}(1/2-i\tau+it_f)d\tau,
        \end{equation}
        where
        $$\gamma^{\mp}(-1/2+i\tau-it_f)= \frac{\Gamma\left(\frac{1/2+i\tau}{2}\right)\Gamma\left(\frac{1/2+i(\tau-2t_f)}{2}\right)}{\Gamma\left(\frac{1/2-i\tau}{2}\right)\Gamma\left(\frac{1/2-i(\tau-2t_f)}{2}\right)}\mp \frac{\Gamma\left(\frac{3/2+i\tau}{2}\right)\Gamma\left(\frac{3/2+i(\tau-2t_f)}{2}\right)}{\Gamma\left(\frac{3/2-i\tau}{2}\right)\Gamma\left(\frac{3/2-i(\tau-2t_f)}{2}\right)}.$$
        \subsubsection{The mildly oscillatory range} Let $T^\varepsilon\ll B=NX/CQ\ll T^{1-\varepsilon}$. Then
        \begin{equation}\label{3.15}
            \begin{split}
                \tilde{g_1}\left(\frac{1}{2}-i(\tau-t_f)\right) &= \int V\left(\frac{z}{N}\right)z^{-1/2-i\tau}e\left(\frac{zx}{qQ}\right)dz\\
                &= N^{1/2-i\tau}\int V(z)z^{-1/2}e\left(\frac{Nxz}{qQ}-\frac{\tau\log z}{2\pi}\right).
            \end{split}
        \end{equation}
        Let $$ h(z)=\frac{Nxz}{qQ}-\frac{\tau\log z}{2\pi}. $$ One has 
        $$ h'(z)= \frac{Nx}{qQ}-\frac{\tau}{2\pi z}, \;\;\; h^{(j)}\asymp \tau,$$ for $j\geq 2$. By repeated integration by parts $\tilde{g_1}$ is negligibly small unless $\tau\asymp NX/CQ=B$, in which case, the stationary point of $h$ is $z_0=\tau qQ/2\pi Nx$, when $$h(z_0)=-\frac{\tau}{2\pi}\log\frac{\tau qQ}{eNx}$$ and $h''(z_0)\asymp \tau\asymp B$. By Lemma \ref{stationary_phase}, upto a negligible error term one has 
        $$ \tilde{g_1}\left(\frac{1}{2}-i(\tau-t_f)\right)\asymp N^{1/2-i\tau}B^{-1/2}V_1\left(\frac{\tau}{B}\right)e\left(-\frac{\tau}{2\pi}\log\frac{\tau qQ}{eNx}\right),$$
        for a compactly supported (in $\mathbb{R}_{>0}$), $T^\eps$-inert function $V_1$ (here we have used Mellin's technique to separate variables and obtain $V_1$). Therefore, 
        \begin{align}\label{3.16}
            G_x^{\pm_1}(y)= N^{1/2}B^{-1/2}
            &(\pi^2y)^{1/2+it_f}\frac{\epsilon_f^{(1\mp 1)/2}}{4\pi^2 i}\\
            \nonumber
            &\int V_1\left(\frac{\tau}{B}\right)\gamma^\mp(-1/2+i\tau-it_f)e\left(-\frac{\tau}{2\pi}\log\frac{\pi^2\tau yqQ}{ex}\right)d\tau+ O_A(T^{-A}).
        \end{align}
        Stirling's approximation for the ratio of gamma factors \eqref{2.18} gives
        \begin{equation}\label{3.17}
            \gamma^{\mp}\left(-1/2+i\tau-i t_f\right)= w^{\mp}(\tau)e\left(\frac{1}{2\pi} \left(\tau \log\frac{\tau}{2e}+(\tau-2t_f)\log\frac{2t_f-\tau}{2e}\right)\right)+ O_A(T^{-A}),
        \end{equation}
        with $$\tau^j\frac{\mathop{d}^{j} w^{\mp}(\tau)}{\mathop{d}\tau^{j}}\ll_j 1,$$ 
        for $T^\varepsilon\ll \tau\ll T^{1-\varepsilon}$. Then, at the cost of a negligible error, 
        \begin{align}\label{3.18}
            G_x^{\pm_1}(y)= N^{1/2}&B^{-1/2}
            (\pi^2y)^{1/2+it_f}\frac{\epsilon_f^{(1\mp 1)/2}}{4\pi^2 i}\\
            \nonumber
            &\int V_1^{\pm}\left(\frac{\tau}{B}\right)e\left(\frac{1}{2\pi} \left(\tau \log\frac{x}{2\pi^2yqQ}+(\tau-2t_f)\log\frac{2t_f-\tau}{2e}\right)\right)d\tau,
        \end{align}
       where $V_1^{\pm}(\tau)=V_1(\tau)w^{\mp}(B\tau).$  Note that  $V_1^\pm(\tau)$ is $T^\varepsilon$-inert and compactly supported in $\mathbb{R}_{>0}$. With the change of variables $\tau \rightsquigarrow B\tau$,
        \begin{align}\label{3.19}
            G_x^{\pm_1}(y)= N^{1/2}B^{1/2}&(\pi^2y)^{1/2+it_f} \frac{\epsilon_f^{(1\mp 1)/2}}{4\pi^2 i}\\
            \nonumber
            &\int V^{\pm}_1(\tau)e\left(\frac{1}{2\pi} \left(B\tau \log\frac{x}{2\pi^2yqQ}+(B\tau-2t_f)\log\frac{2t_f-B\tau}{2e}\right)\right)d\tau.
        \end{align}
        Define 
        $$h(\tau)= B\tau \log\frac{x}{2\pi^2yqQ}+(B\tau-2t_f)\log\frac{2t_f-B\tau}{2e}.$$ Then,
        $$h'(\tau)= B\log \frac{xB(2t_f-B\tau)}{4\pi^2XNY}\gg B,$$ unless $Ny\asymp BT$ and
        $$h^{(j)}(\tau)\asymp_j \frac{T}{(T/B)^j},\;\; j\geq 2.$$ With $X_0=V_0=1$, $Q_0=T/B$, $Y_0=T$, $R_0=B$, Lemma $\eqref{repeated_integration_by_parts}$ implies $G^{\pm_1}(y)$ is negligibly small unless $Ny\asymp BT$.
        
        Similar analysis, with an initial change of variable $\tau \rightsquigarrow -\tau$ in $G^{\pm_2}$ yields, for a compactly supported (in $\mathbb{R}_{>0}$) $T^\eps$-inert function $U^{\pm}_2$,
        \begin{align}\label{3.20}
            G_x^{\pm_2}(y)= -N^{1/2}&B^{1/2}(\pi^2y)^{1/2+it_g}\frac{\epsilon_g^{(1\mp 1)/2}}{4\pi^2 i}\\
            \nonumber
            &\int U^{\pm}_2(\tau)e\left(\frac{1}{2\pi} \left(-B\tau \log\frac{x}{2\pi^2yqQ}-(B\tau+2t_g)\log\frac{2t_g+B\tau}{2e}\right)\right)d\tau,
        \end{align}
         if $Ny \asymp  BT$ and negligibly small otherwise. Henceforth, by abuse of notation, we will simply write $V_1(\tau)$ and $U_2(\tau)$ in place $V_1^{\pm}(\tau)$ and $U_2^{\pm}(\tau)$ respectively, as the only property of these functions which is to be mindful of is that they are compactly supported and $T^\eps$-inert.
         \subsubsection{The highly oscillatory case} Finally let
            $$B= \frac{NX}{CQ}\gg T^{1-\varepsilon}.$$ Carrying out the evaluation of the Mellin transform $\tilde{g_1}$ as in \eqref{3.15}  we again obtain the expression
            $$ \tilde{g_1}\left(\frac{1}{2}-i(\tau-t_f)\right)\asymp N^{1/2-i\tau}B^{-1/2}V_1\left(\frac{\tau }{B}\right)e\left(-\frac{\tau}{2\pi}\log\frac{\tau qQ}{eNx}\right).$$ Then with the change of variables $\tau\rightsquigarrow B\tau$ we have 
            \begin{align}\label{3.21}
                G_x^{\pm_1}(y)= N^{1/2}B^{1/2}
                &(\pi^2y)^{1/2+it_f}\\
                \nonumber
                &\int V_1\left(\tau\right)\gamma^{\mp}(-1/2+iB\tau-it_f)e\left(-\frac{\tau}{2\pi}\log\frac{\pi^2B\tau yqQ}{ex}\right)d\tau,
            \end{align}
            with $$\gamma^{\mp}(-1/2+iB\tau-it_f)= \frac{\Gamma\left(\frac{1/2+iB\tau}{2}\right)\Gamma\left(\frac{1/2+i(B\tau-2t_f)}{2}\right)}{\Gamma\left(\frac{1/2-iB\tau}{2}\right)\Gamma\left(\frac{1/2-i(B\tau-2t_f)}{2}\right)}\mp \frac{\Gamma\left(\frac{3/2+iB\tau}{2}\right)\Gamma\left(\frac{3/2+i(B\tau-2t_f)}{2}\right)}{\Gamma\left(\frac{3/2-iB\tau}{2}\right)\Gamma\left(\frac{3/2-i(\tau-2t_f)}{2}\right)}.$$
            Stirling approximation \eqref{2.18} on
            $$\frac{\Gamma\left(\frac{1/2+iB\tau}{2}\right)}{\Gamma\left(\frac{1/2-iB\tau}{2}\right)}\;\; \text{and}\;\; \frac{\Gamma\left(\frac{3/2+iB\tau}{2}\right)}{\Gamma\left(\frac{3/2-iB\tau}{2}\right)}$$implies that at the cost of an error term of size at most $O_A(T^{-A})$,  $G_x^{\pm_1}(y)$ is equal to
            \begin{align}\label{3.23}
                N^{1/2}B^{1/2}&(\pi^2y)^{1/2+it_f}\\ \nonumber
                &\int V_{1}(\tau) \omega^\pm_{1}(\tau)e\left(\frac{1}{2\pi} \left(B\tau \log\frac{x}{2\pi^2yqQ}\right)\right)\mathop{d\tau},
            \end{align}
            where $$\omega^\pm_{1}(\tau)= w_{1/2,A}(B\tau)\frac{\Gamma\left(\frac{1/2+i(B\tau-2t_f)}{2}\right)}{\Gamma\left(\frac{1/2-i(B\tau-2t_f)}{2}\right)}\mp w_{3/2,A}(B\tau)\frac{\Gamma\left(\frac{3/2+i(B\tau-2t_f)}{2}\right)}{\Gamma\left(\frac{3/2-i(B\tau-2t_f)}{2}\right)}\ll 1.$$
             The functions $w_{1/2,A},w_{3/2,A}$ are as in \eqref{2.18}. Similar analysis yields that upto a negligible error term $G_x^{\pm_2}(y)$ equal to,
            \begin{align}
                -N^{1/2}B^{1/2}&(\pi^2y)^{1/2+it_g}\\ \nonumber
                &\int U_2(\tau) \omega^\pm_{2}(\tau)e\left(\frac{1}{2\pi} \left(-B\tau \log\frac{x}{2\pi^2yqQ}\right)\right)\mathop{d\tau},
            \end{align}
            where $$\omega^\pm_2(\tau)\ll 1.$$
             We record the above observations in the following Lemma.
            \begin{Lemma}\label{structure_of_G}
                For $j=1,2$, one has, at the cost of negligible error terms,
                \begin{enumerate}[label=\roman*)]
                    \item If $B\ll T^\eps$, $G_x^{\pm_{j}}(y)$ is negligibly small unless $Ny\ll_\eps T^{1+\eps}$ in which case
                    \begin{align*}
                        G_x^{\pm_j}(y)\ll_\eps (Ny)^{1/2}T^\eps.
                    \end{align*}
                    \item If $T^\varepsilon\ll B\ll T^{1-\varepsilon}$, then $G^{\pm_j}(y)$ is negligibly small unless $Ny\asymp BT$, in which case 
                    \begin{align*}
                        G_x^{\pm_1}(y)= N^{1/2}B^{1/2}&(\pi^2y)^{1/2+it_f} \frac{\epsilon_f^{(1\mp 1)/2}}{4\pi^2 i}\\
                        &\int V_1(\tau)e\left(\frac{1}{2\pi} \left(B\tau \log\frac{x}{2\pi^2yqQ}+(B\tau-2t_f)\log\frac{2t_f-B\tau}{2e}\right)\right)d\tau,
                    \end{align*}
                    and 
                    \begin{align*}
                        G_x^{\pm_2}(y)= -N^{1/2}&B^{1/2}(\pi^2y)^{1/2+it_g}\frac{\epsilon_g^{(1\mp 1)/2}}{4\pi^2 i}\\
                        &\int U_2(\tau)e\left(\frac{1}{2\pi} \left(-B\tau \log\frac{x}{2\pi^2yqQ}-(B\tau+2t_g)\log\frac{2t_g+B\tau}{2e}\right)\right)d\tau,
                    \end{align*}
                    where $V_1(\tau),U_2(\tau)$ are $T^\eps$-inert, compactly supported in $\mathbb{R}_{>0}$, smooth functions.
                    
                    \item If $T^{1-\varepsilon}\ll B$, then $G_x^{\pm_j}(y)$ is negligibly small unless $Ny\ll_\varepsilon T^\varepsilon B^2$, in which case
                    \begin{align*}
                        G_x^{\pm_1}(y)=  N^{1/2}B^{1/2}&(\pi^2y)^{1/2+it_f}\frac{\epsilon_f^{(1\mp 1)/2}}{4\pi^2 i}\\ \nonumber
                         &\int V_{1}(\tau) \omega^\pm_{1}(\tau)e\left(\frac{1}{2\pi} \left(B\tau \log\frac{x}{2\pi^2yqQ}\right)\right)\mathop{d\tau},
                    \end{align*}
                    and 
                    \begin{align*}
                         G_x^{\pm_2}(y)= -N^{1/2}B^{1/2}&(\pi^2y)^{1/2+it_g}\frac{\epsilon_g^{(1\mp 1)/2}}{4\pi^2 i}\\ \nonumber
                        &\int U_2(\tau) \omega^\pm_{2}(\tau)e\left(\frac{1}{2\pi} \left(-B\tau \log\frac{x}{2\pi^2yqQ}\right)\right)\mathop{d\tau},
                    \end{align*}
                     where $$\omega^\pm_{1}(\tau), \omega^\pm_{2}(\tau)\ll 1$$ and  $V_{1}(\tau), U_{2}(\tau)$ $T^\eps$-inert, compactly supported in $\mathbb{R}_{>0}$ and smooth.
                \end{enumerate}
            \end{Lemma}
            \section{The Ramanujan Sum}\label{Ramanujan_sum}
            We now focus on the $a$-sum obtained after the Voronoi Summation Formula, which is a Ramanujan's sum with modulus $q$, 
            \begin{align}\label{4.1}
                c_q(\mp n\pm m)=\sideset{}{^*}\sum_{a\bmod q}e\left(\frac{(\mp n\pm m)\bar{a}}{q}\right).
            \end{align}
            We will use the following identity for the $ c_q(\mp n\pm m)$.
            \begin{align}\label{4.2}
                c_q(\mp n\pm m)= \sum_{d|q}d\mu\left(\frac{q}{d}\right)\mathds{1}_{\mp n\pm m\equiv 0 \bmod d}.
            \end{align}
            Using the above idenity and Lemma \ref{structure_of_G} in   \eqref{2.29} and taking a smooth dyadic partition of the $m,n$-sums we obtain,
            \begin{align}\label{4.3}
                S_{+}(N,C,X)\ll_{A,\eps} \sup_{\substack{N'_0\ll N_1\ll N_0\\
                M'_0\ll M_1\ll M_0}}\frac{T^\varepsilon}{Q}&\sum_{q\sim C}q\sum_{d|q}d\mu\left(\frac{q}{d}\right)\\
                 &\nonumber\sum_{\pm} \sum _{n\sim N_1}\frac{\lambda_f(n)}{n}\sum_{\substack{m\sim M_1\\  m\equiv \pm n\bmod d}}\frac{\lambda_g(m)}{m}I(m,n,q)+ T^{-A},
            \end{align} 
            where 
            \begin{align}
                I(m,n,q)=\int W\left(\frac{x}{X}\right)g(q,x)G^{\pm_1}_x\left(\frac{n}{q^2}\right)G^{\pm_2}_x\left(\frac{m}{q^2}\right)
            \end{align}
            and
            \begin{enumerate}
                \item $N_0'=M_0'=1$, $N_0=M_0=T^{1+\varepsilon}C^2/N$, if $B\ll T^\varepsilon$.
                \item $N_0',M_0',N_0,M_0= BTC^2/N$, if $T^\varepsilon\ll B\ll T^{1-\eps}$.
                \item $N_0'=M_0'=1$, $N_0=M_0= T^\eps B^2C^2/N$, if $T^{1-\eps}\ll B$.
            \end{enumerate}
        \section{Simplification of \texorpdfstring{$I(m,n,q)$}{I(m,n,q)}}\label{simplification_of_I}
        In this section we simplify  $I(m,n,q)$ further.
       
       If $B\ll T^\varepsilon$, using the estimate \eqref{3.12} for $G_x^{\pm_{1,2}}(\tau)$ and estimating the $x$-integral trivially we obtain
        \begin{align}
            I(m,n,q)\ll_\eps \frac{T^\eps XN(mn)^{1/2}}{q^2}.
        \end{align}
        If $ T^\varepsilon\ll B\ll T^{1-\eps}$, $I(m,n,q)$ equals
        \begin{align}\label{5.2}
            NB\left(\frac{\pi^2n}{q^2}\right)^{1/2+it_f}\frac{\epsilon_f^{(1\mp 1)/2}}{4\pi^2 i}&\left(\frac{\pi^2m}{q^2}\right)^{1/2+it_g}\frac{\epsilon_g^{(1\mp 1)/2}}{4\pi^2 i}\mathcal{I}(m,n,q),
        \end{align}
        where 
        \begin{align}
            \mathcal{I}(m,n,q)=\int W\left(\frac{x}{X}\right)g(q,x)  \int\int V_1(\tau_1)U_2(\tau_2)e\left(\frac{1}{2\pi}h(\tau_1,\tau_2,x)\right)\mathop{d\tau_1 d\tau_2 dx},
        \end{align}
        with
        \begin{align*}
            h(\tau_1,\tau_2,x)&=  B\tau_1 \log\frac{xq}{2\pi^2nQ}+(B\tau_1-2t_f)\log\frac{2t_f-B\tau_1}{2e}\\
            &-B\tau_2 \log\frac{xq}{2\pi^2mQ}-(B\tau_2+2t_g)\log\frac{2t_g+B\tau_2}{2e}.
        \end{align*}
        Consider the $x$-integral
        \begin{align}\label{5.3}
            I_x= \int W\left(\frac{x}{X}\right)g(q,x)e\left(\frac{B\log x}{2\pi}(\tau_1-\tau_2)\right).
        \end{align}
        We change variable $x\rightsquigarrow Xx$ to get 
        \begin{align}
            I_x=Xe\left(\frac{B\log X}{2\pi}(\tau_1-\tau_2)\right)\int W(x)g(q,Xx) e\left(\frac{B\log x}{2\pi}(\tau_1-\tau_2)\right)\mathop{dx}.
        \end{align}
        Note that either upto a negligible error term (of size $O_A(T^{-A})$) $g(q,x)$ can be replaced by $1$ or $x^j \frac{\partial^j}{\partial x^j}g(q,x)\ll_{j,\eps} Q^{(j+1)\varepsilon}$. Therefore denoting $W_q(x)=W(x)g(q,Xx)$ gives
        \[I_x=Xe\left(\frac{B\log X}{2\pi}(\tau_1-\tau_2)\right)\int W_q(x) e\left(\frac{B\log x}{2\pi}(\tau_1-\tau_2)\right)\mathop{dx}+ O_A(T^{-A}),\] where $\frac{d^j}{dx^j}W_q(x)\ll_{j,\eps} Q^{(j+1)\eps}$. Let (temporarily) $h(x)= B\log x(\tau_1-\tau_2)/2\pi$. Then
        \[h^{(j)}(x)\asymp_j {B(\tau_1-\tau_2)}.\] With $R_0=Y_0=B|\tau_1-\tau_2|$, $X_0= Q_0=1$, $V_0=Q^{-2\eps}$ Lemma \ref{repeated_integration_by_parts} implies that $I_x$ is negligibly small unless
        \begin{align*}
            \tau_1-\tau_2\ll \frac{T^\varepsilon}{B}.
        \end{align*}
        writing $\tau_2=\tau_1+u$,  with $u\ll T^\eps B^{-1}$, 
        \begin{align}\label{5.6}
            \mathcal{I}(m,n,q)= X\int_{u\ll T^\eps B^{-1}}\int W_q(x)e\left(-\frac{Bu\log x}{2\pi}\right) \mathcal{I}_u(m,n,q)\mathop{dx}\mathop{du},
        \end{align}
        \begin{align*}
             \mathcal{I}_u(m,n,q)=  \int V_{u}(\tau_1)e\left(\frac{1}{2\pi}h(\tau_1,\tau_1+u,X)\right)\mathop{d\tau_1},
        \end{align*}
        $V_u(\tau_1)=V_1(\tau_1)U_2(\tau_1+u)$, which is $T^\eps$-inert and compactly supported in $\mathbb{R}_{>0}$ as a function of $\tau_1$. Now,
        \begin{align}
             h(\tau_1+u,\tau_1,X)&= B\tau_1\log\frac{m}{n}+(B\tau_1-2t_f)\log\frac{2t_f-B\tau_1}{2e}\\ \nonumber
             &-(B(\tau_1+u)+2t_g)\log\frac{2t_g+B(\tau_1+u)}{2e}-Bu\log\frac{Xq}{2\pi^2mQ}.
        \end{align}
        We have
        \begin{align*}
            (B(\tau_1+u)+2t_g)\log\frac{2t_g+B(\tau_1+u)}{2e}&= (B\tau_1+2t_g)\log\frac{2t_g+B\tau_1}{2e}+ H_{1,u}(\tau_1)+H_{2,u}(\tau_1),
        \end{align*}
        where $$H_{1,u}(\tau_1)= Bu\log\frac{2t_g+B(\tau_1+u)}{2e}$$ and $$H_{2,u}(\tau_1)= (B\tau_1+2t_g)\log \left(1+\frac{Bu}{B\tau_1+2t_g}\right),$$
       A routine computation using Taylor series expansion implies that 
        \[\frac{\partial^j}{\partial \tau_1^j}H_{1,u}(\tau_1), \frac{\partial^j}{\partial \tau_1^j}H_{2,u}(\tau_1) \ll_{\eps,j}T^{j\eps}.\]
        Defining $V_{2,u}(\tau_1)=V_u(\tau_1)e\left(\frac{1}{2\pi}(H_{1,u}(\tau_1)+H_{2,u}(\tau_1)\right)$ which is again $T^\eps$- inert and compactly supported in $\mathbb{R}_{>0}$,
        \begin{align}\label{5.8}
            \mathcal{I}_u(m,n,q)=e\left(-Bu\log\frac{Xq}{2\pi^2mQ}\right)\int V_{2,u}(\tau_1)e\left(\frac{1}{2\pi}H(\tau_1)\right), 
        \end{align}
        with \[H(\tau_1)=B\tau_1\log\frac{m}{n}+(B\tau_1-2t_f)\log\frac{2t_f-B\tau_1}{2e}-(B\tau_1+2t_g)\log\frac{2t_g+B\tau_1}{2e}.\]
        
        If $B\gg T^{1-\varepsilon}$, we can deduce the same expression for $\mathcal I(m,n,q)$ as in \eqref{5.2} with 
        \begin{align}
            \mathcal{I}_u(m,n,q)=\int V_u(\tau_1)\omega^\pm_1(\tau_1)\omega_2^{\pm}(\tau_1+u)e\left(\frac{1}{2\pi}h(\tau_1,\tau_1+u,X)\right),
        \end{align}
        $V_u(\tau_1)=V(\tau_1)U(\tau_1+u),$ and 
        \begin{align*}
            h(\tau_1,\tau_2,x)= B\tau_1\log \frac{xq}{2\pi^2nQ}-B\tau_2\log \frac{xq}{2\pi^2mQ}.
        \end{align*}
        Similar analysis leading upto \eqref{5.8}, in this case as well, implies
        \begin{align}
            \mathcal{I}_u(m,n,q)=e\left(-Bu\log\frac{Xq}{2\pi^2mQ}\right)\int V_{u}(\tau_1)\omega^\pm_1(\tau_1)\omega_2^{\pm}(\tau_1+u)e\left(\frac{1}{2\pi}H(\tau_1)\right),
        \end{align}
        where $$H(\tau_1)= B\tau_1\log \frac{m}{n},$$
        $V_u(\tau_1)$ is $T^\varepsilon$-inert, compactly supported on $\mathbb{R}_{>0}$ as a function of $\tau$ and $\omega^\pm_1(\tau_1)\omega_2^{\pm}(\tau_1+u)\ll 1.$ We compile the above observations in the following Lemma.
        \begin{Lemma}\label{structure_of_I}
            We obtain, upto negligible error terms of size $O_A(T^{-A})$
            \begin{enumerate}[label=(\arabic*)]
                \item If $B \ll T^\varepsilon$, then
                \begin{align}
                    I(m,n,q)\ll_\eps \frac{T^\eps XN(mn)^{1/2}}{q^2}.
                \end{align} \label{item5.1_1}
                \item If $T^\eps \ll B$ then,
                \begin{align}\label{5.12}
                    I(m,n,q)= NB\left(\frac{\pi^2n}{q^2}\right)^{1/2+it_f}\frac{\epsilon_f^{(1\mp 1)/2}}{4\pi^2 i}&\left(\frac{\pi^2m}{q^2}\right)^{1/2+it_g}\frac{\epsilon_g^{(1\mp 1)/2}}{4\pi^2 i}\mathcal{I}(m,n,q) ,
                \end{align}
                where\label{item5.1_2}
            \begin{enumerate}[label=(\roman*)]
                \item If $B\ll T^{1-\varepsilon}$
                \begin{align}
                    \mathcal{I}(m,n,q)= X\int_{u\ll T^\varepsilon B^{-1}}&\int W_q(x)e\left(-\frac{Bu}{2\pi}\log\frac{xXq}{2\pi^2mQ}\right)\mathcal{I}^\sharp_u(m,n)\mathop{dx}\mathop{du},
                \end{align}
                with
                \[\mathcal{I}^\sharp_u(m,n)=\int V_u(\tau) e\left(\frac{1}{2\pi}H(\tau) \right)\mathop{d\tau}\]
                where $W_q(x), V_u(\tau)\ll_\eps 1$ are compactly supported, smooth, $T^\eps$-inert functions and $$H(\tau)=B\tau\log\frac{m}{n}+(B\tau-2t_f)\log\frac{2t_f-B\tau}{2e}-(B\tau+2t_g)\log\frac{2t_g+B\tau}{2e}.$$\label{item5.1_2_1}
                \item
                If $B\gg T^{1-\varepsilon}$ 
                \begin{align}
                    \mathcal{I}(m,n,q)= X\int_{u\ll T^\varepsilon B^{-1}}&\int W_q(x)e\left(-\frac{Bu}{2\pi}\log\frac{xXq}{2\pi^2mQ}\right)\mathcal{I}_u^\sharp(m,n)\mathop{dx}\mathop{du},
                \end{align}
                with
                \[\mathcal{I}_u^\sharp(m,n)= \int V_u(\tau) \omega_1^\pm(\tau)\omega_2^\pm(\tau+u) e\left(\frac{1}{2\pi}H(\tau) \right)\mathop{d\tau}\]
                where $W_q(x), V_u(\tau)$ are compactly supported and $T^\eps$-inert functions, $\omega_1^\pm(\tau)\omega_2^\pm(\tau+u)\ll 1$ and
                $$H(\tau)=B\tau\log\frac{m}{n}.$$\label{item5.1_2_2}
            \end{enumerate}
            \end{enumerate}
        \end{Lemma}
       In the next Lemma, using Lemma \ref{stationary_phase}, we determine precise estimates on $\mathcal I_u^\sharp(m,n)$ when $T^\eps\ll B\ll T^{1-\eps}$. When $B\gg T^{1-\eps}$ we  estimate $\mathcal I_u^\sharp(m,n)$ in the $L^2$ sense.
        \begin{Lemma}\label{structure_of_I_sharp}
            Let $t_{f,g}=t_f/t_g\asymp 1$. One has 
            \begin{enumerate}[label=(\arabic*)]
                \item If $T^\eps\ll B\ll T^{1/2+\eps}$, then $\mathcal{I}^\sharp_u(m,n)$ is negligibly small unless \[n-m t_{f,g}\ll N_2= \frac{N_0T^\eps}{B}.\]\label{item5.2_1}
                \item If $T^{1/2+\eps}\ll B\ll T^{1-\eps}$, then $\mathcal I_u^\sharp(m,n)$ is negligibly small unless \[n-mt_{f,g}\asymp N_2= \frac{N_0B}{T},\] in which case 
                \[\mathcal I_u^\sharp(m,n)= \frac{T^{1/2}}{B}V_{\eps,u}\left(\frac{n-m t_{f,g}}{N_2}\right)e\left(\frac{1}{2\pi}H(n,m)\right) +O_A(T^{-A}),\]
                 where $V_{\eps,u}(\tau)$ is a compactly supported  and $T^\eps$- inert as a function of $\tau$ with\\ $V_{\eps,u}(\tau)\ll _\eps 1$ and  \[H(n,m)= 2t\log \frac{e(n+m)}{t}-2t_f\log n-2t_g\log m.\]\label{item5.2_2}
                 \item If $B\gg T^{1-\eps}$, then one has 
                 \begin{align}
                     \int \varphi(w)\left|\mathcal{I}_u^\sharp(m,N_1w)\right|^2\mathop{dw}\ll_\eps\frac{T^\eps}{B},
                 \end{align}
                 for a compactly supported $T^\eps$-inert function $\varphi$.\label{item5.2_3}
            \end{enumerate}
        \end{Lemma}
        \begin{proof}
         The derivatives of $H$ satisfy 
        \begin{align}
            \frac{d}{d\tau_1}H(\tau) &= B \left(\log \frac{m}{n}+\log(2t_f-B\tau)-\log(2t_g+B\tau)\right)\\ \nonumber
            &= B\left(\log \frac{mt_f}{nt_g}+H_2(\tau)\right),
        \end{align}
        where $$H_2(\tau)= \log \left(1-\frac{B\tau}{2t_f}\right)-\log\left(1+\frac{B\tau}{2t_g}\right)\asymp \frac{B}{T}$$ and
        
        \begin{align}
            \frac{d^2}{d\tau^2}H(\tau)= B\frac{d^2}{d\tau^2}H_2(\tau)\asymp \frac{B^2}{T},\;\; \frac{d^j}{d\tau^j}H(\tau)= B\frac{d^j}{d\tau^2}H_2(\tau)\ll_j\frac{T}{\left(\frac{T}{B}\right)^j}.
        \end{align} 
        When $B\ll T^{1/2+\varepsilon}$, Lemma \ref{repeated_integration_by_parts} with $X_0=1$, $V_0=T^{-\eps}$ $Y_0=T$, $Q_0=T/B$ and $R_0=T^\varepsilon$, implies that $\mathcal{I}_u(m,n,q)$ is negligibly small unless
        \begin{align}\label{5.18} 
        \log\frac{mt_f}{nt_g}&\ll \frac{T^\varepsilon}{B}\\ \nonumber
        \implies n-m\frac{t_f}{t_g}&\ll N_2= \frac{N_0 T^\varepsilon}{B},\;\; \text{from the mean value theorem.}
        \end{align} 
        When $B\gg T^{1/2+\eps}$, Lemma \ref{repeated_integration_by_parts} with $X_0=1$, $V_0=T^{-\eps}$, $Y_0=T$, $Q_0=T/B$ and $R_0= B^2/T$ implies that $\mathcal{I}_u(m,n,q)$ is negligibly small unless 
        \begin{align}\label{5.19}
            \log\frac{mt_f}{nt_g}&\asymp \frac{B}{T}\\ \nonumber
            \implies n-m \frac{t_f}{t_g}&\asymp N_2= \frac{N_0B}{T},\;\; \text{again from the mean value theorem}.
        \end{align} 
        Furthermore
        \begin{align*}
            H'(\tau_0)=0\;\; \text{at}\;\; \tau_0=\frac{2m t_f-2nt_g}{B(m+n)},
        \end{align*} 
        \[H(\tau_0)= 2t\log \frac{e(n+m)}{t}-2t_f\log n-2t_g\log m.\]
        Then Proposition \ref{stationary_phase} with $Y_0=T$, $Q_0= T/B$, $X_0=\Omega=1$,  $V_0=T^{-\eps}$ and $\delta = \eps/100$ followed by Mellin's technique to separate variables in the weight function gives us 
        \begin{align}
            \mathcal{I}_u^\sharp(m,n)= \frac{T^{1/2}}{B}V_{\eps,u}\left(\frac{n-mt_{f,g}}{N_2}\right)e\left(\frac{1}{2\pi}H(n,m)\right) +O_A(T^{-A}),
        \end{align}
        where 
        \[H(n,m)= 2t\log \frac{e(n+m)}{t_f+t_g}-2t_f\log n-2t_g\log m,\]
        $V_{\eps,u}(\tau)\ll_\eps 1$ being compactly supported in $\mathbb{R}\setminus\{0\}$ and $T^\eps$-inert. 
        
        When $B\gg T^{1-\eps}$, $ |\mathcal{I}_u(m,N_1w)|^2$ equals
        \begin{align*}
           \int&\int V_u(\tau_1) \omega_1^\pm(\tau_1)\omega_2^\pm(\tau_1+u)\overline{V_u(\tau_2) \omega_1^\pm(\tau_2)\omega_2^\pm(\tau_2+u)} e\left(\frac{B}{2\pi}\log\frac{m}{N_1w}\left(\tau_1-\tau_2\right)\right)\mathop{d\tau_1}\mathop{d\tau_2}.
        \end{align*}
        Ignoring the $\tau_1,\tau_2$ integrals briefly and isolating the $w$-integral leads to
        \begin{align*}
               \int \varphi(w)e\left(\frac{B}{2\pi}(\tau_1-\tau_2)\log w\right).
        \end{align*}
        Then again Lemma \ref{repeated_integration_by_parts} implies the the $w$-integral is negligibly small unless $\tau_1-\tau_2\ll T^\eps/B$, in which case,
        \begin{align*}
                \int \varphi(w)\left|\mathcal{I}_u^\sharp(m,N_1w)\right|^2\mathop{dw}&\ll_A \int\int _{\tau_1-\tau_2\ll \frac{T^\eps}{B}}(\cdots)\mathop{d\tau_1d\tau_2 dw}+ T^{-A}
                \ll \frac{T^\eps}{B},
        \end{align*}
        completing the proof of the Lemma.
        \end{proof}
       
    \section{Cauchy's inequality on \texorpdfstring{$n$}{n}}\label{Cauchy}
        Next in order to break the involution we eliminate the $GL(2)$ Fourier coefficients involving $f$. We apply Cauchy's inequality on $n$. With this it is better to settle estimating $S_+(N,C,X)$ when $B\ll T^\eps.$ Using \ref{item5.1_1} of Lemma \ref{structure_of_I} on $I(m,n,q)$ we obtain by triangle inequality
        \begin{align}
            S_+(N,C,X)\ll_{A,\eps} \frac{T^\eps NX}{Q}\sup_{\substack{N'_0\ll N_1\ll N_0\\
                M'_0\ll M_1\ll M_0}} &\sum_{q\sim C}\frac{1}{q}\sum_{d|q}d\\ \nonumber
                &\sum_{\pm} \sum_{n\sim N_1} \frac{|\lambda_f(n)|}{n^{1/2}}\sum_{\substack{m\sim M_1\\ m\equiv \pm n\bmod d}}\frac{|\lambda_g(m)|}{m^{1/2}}+T^{-A}.
        \end{align}
       Using Cauchy's inequality on $n$, the $n,m$-sum is bounded by 
      $\Theta^{1/2}\Omega_d^{1/2}$, where
      \begin{align*}
          \Theta=\sum_{n\sim N_1}\frac{|\lambda_f(n)|^2}{n}\ll_\eps T^\eps,
      \end{align*}
      by Ramanujan bound on average \eqref{2.2} and
      \begin{align*}
          \Omega_d = \sum_{n\sim N_1}\left|\sum_{\substack{m\sim M_1\\ m\equiv \pm n \bmod d}}\frac{|\lambda_g(m)|}{m^{1/2}}\right|^2.
      \end{align*}
      Opening up the absolute value square inside the $n$ sum leads to
      \begin{align*}
         \Omega_d&=  \sum_{\substack{m_1,m_2\sim M_1 \\ m_1\equiv m_2 \bmod d}} \frac{|\lambda_g(m_1)|}{m_1^{1/2}}\frac{|\lambda_g(m_2)|}{m_2^{1/2}}\sum_{\substack{n\sim N_1\\n\equiv \pm m_1 \bmod d}}1\\ &\ll \left(\frac{N_1}{d}+1\right)\sum_{\substack{m_1,m_2\sim M_1 \\ m_1\equiv m_2 \bmod d}} \frac{|\lambda_g(m_1)|}{m_1^{1/2}}\frac{|\lambda_g(m_2)|}{m_2^{1/2}}.
      \end{align*}
      Since the sum is symmetric with respect to $m_1$ and $m_2$ it is enough to estimate the contribution due to $m_2 \geq m_1$. Writing $m_2=m_1+rd$ with $0\leq r\ll M_1/d$ the $m_1,m_2$-sum reduces to
      \begin{align*}
          &\sum_{m_1\sim M_1}\frac{|\lambda_g(m_1)|}{m_1^{1/2}}\sum_{0\leq r\ll \frac{M_1}{d}}\frac{|\lambda_g(m_1+rd)|}{(m_1+rd)^{1/2}}\\ 
          &\ll \left(\sum_{0\leq r\ll \frac{M_1}{d}}\sum_{m\sim M_1}\frac{|\lambda_g(m_1)|^2}{m_1}+\sum_{0\leq r\ll \frac{M_1}{d}}\sum_{m\sim M_1}\frac{|\lambda_g(m_1+rd)|^{2}}{m_1+rd}\right)\ll_\eps T^\eps\left(1+\frac{M_1}{d}\right),
      \end{align*}
      by A.M-G.M inequality and Ramanujan bound on average \eqref{2.2} on the $m_1$ sum. Therefore
      \begin{align*}
          \Omega_d \ll_\eps T^\eps \left(\frac{N_1}{d}+1\right)\left(1+\frac{M_1}{d}\right)\ll_\eps T^\eps\left(1+\frac{N_0}{d}\right)^2,
      \end{align*}
      after taking supremum on $N_1$ and $M_1$ and using $N_0=M_0$. Hence 
      \begin{align*}
          S_+(N,C,X)&\ll_{A,\eps} \frac{T^\eps NX}{Q}\sum_{q\sim C}\frac{1}{q}\sum_{d|q}d\left(1+\frac{N_0}{d}\right) +T^{-A}\\ 
          &\ll_{A,\eps} \frac{T^\eps NX}{Q}\left(C+N_0\right)+ T^{-A},
      \end{align*}
      with $NX/Q=BC$, $N_0=T^{1+\eps}C^2/N$, $C\ll Q$, $B\ll T^\eps$ leading to
      \begin{proposition}\label{prop_6.1}
          If $B=NX/CQ\ll T^\eps$, then
          \begin{align}
              S_+(N,C,X)\ll_{\eps} T^\eps \left(\frac{N}{K}+ \frac{TN^{1/2}}{K^{3/2}}\right).
          \end{align}
      \end{proposition}
      Now if $B\gg T^\eps$ using \eqref{5.12} again by triangle inequality we have
      \begin{align}\label{6.4}
          S_+(N,C,X) \ll_{\eps,A} \frac{T^\eps NB}{Q}\sup_{\substack{N'_0\ll N_1\ll N_0\\
                M'_0\ll M_1\ll M_0}}&\sum_{q\sim C}\frac{1}{q}\sum_{d|q}d\\ \nonumber
                &\sum_{\pm}\sum_{n\sim N_1}\frac{|\lambda_f(n)|}{n^{1/2}}\left|\sum_{\substack{m\sim M_1\\ m\equiv \pm n \bmod d}}\frac{\lambda_g(m)}{m^{1/2-it_g}}\mathcal{I}(m,n,q)\right| +T^{-A}.
      \end{align}
      Writing $q=dl$, taking a dyadic subdivision in the $d$-sum we obtain
      \begin{align}\label{6.5}
          S_+(N,C,X) \ll_{\eps,A} \frac{T^\eps NB}{Q}\sup_{\substack{N'_0\ll N_1\ll N_0\\
                M'_0\ll M_1\ll M_0\\ 1\ll D\ll C}}&\sum_{1\leq l\ll \frac{C}{D}}\frac{1}{l}\sum_{d\sim D}\\ \nonumber
                &\sum_{\pm}\sum_{n\sim N_1}\frac{|\lambda_f(n)|}{n^{1/2}}\left|\sum_{\substack{m\sim M_1\\ m\equiv \pm n \bmod d}}\frac{\lambda_g(m)}{m^{1/2-it_g}}\mathcal{I}(m,n,dl)\right| + T^{-A}.
      \end{align}
      By Cauchy's inequality on the $n$-sum we again obtain that it is bounded by
      \begin{align}\label{6.6}
          \Theta^{1/2}\Omega_{d}^{1/2},
      \end{align}
      where again 
      \begin{align}\label{6.7}
          \Theta=\sum_{n\sim N_1}\frac{|\lambda_f(n)|^2}{n}\ll_\eps T^\eps,
      \end{align}
      by Ramanujan bound on average \eqref{2.2} and 
      \begin{align}\label{6.8}
           \Omega_d = \sum_{n\sim N_1}\left|\sum_{\substack{m\sim M_1\\ m\equiv \pm n \bmod d}}\frac{\lambda_g(m)}{m^{1/2-it_g}} \mathcal{I}(m,n,dl)\right|^2.
      \end{align}
    \section{Poisson on \texorpdfstring{$n$}{n} and Final estimates}\label{Poisson}
    This section is dedicated primarily to the estimation of $\Omega_d$. The principle strategy is to open up the absolute value square, which results in two copies ($m_1$,$m_2$) of the $m$-sum and apply Poisson summation formula to the $n$-sum. The $m_1,m_2$-sum is symmetric and it is enough to estimate the contribution due to $m_2\geq m_1$ .
    \subsection{Oscillation upto \texorpdfstring{$T^{1/2+\eps}$}{T^{1/2}}}\label{Estimate_upto_square_root}
    If $T^\eps\ll B\ll T^{1/2+\eps}$ we estimate $\Omega_d$ using Lemma \ref{structure_of_I_sharp} \ref{item5.2_1}. Using the same in \eqref{6.8}, bounding the $u$ and $x$-integrals trivially, implies
    \begin{align}\label{7.1}
        \Omega_d\ll_{\eps,A}\frac{T^\eps X^2}{B^2}\sum_{n\sim N_1}\left|\sum_{\substack{m\sim M_1\\ m\equiv \pm n\bmod d\\ n-mt_{f,g}\ll N_2}}\frac{\lambda_g(m)}{m^{1/2-it_g}}\right|^2 + T^{-A},
    \end{align} 
    where $N_1,M_1\asymp N_0$, $N_2=N_0T^\eps/B$.
    Opening up the absolute value square, using the symmetry of the $m_1,m_2$-sum and taking the $n$- sum inside, we have
    \begin{align}\label{7.2}
        \Omega_d&\ll_{\eps,A} \frac{T^\eps X^2}{B^2}\sum_{\substack{m_1,m_2\sim M_1\\0\leq m_2-m_1 \\ m_1\equiv m_2\bmod d}}\frac{|\lambda_g(m_1)\lambda_g(m_2)|}{m_1^{1/2}m_2^{1/2}}\sum_{\substack{n-m_1t_{f,g}\ll N_2\\n-m_2t_{f,g}\ll N_2\\ n\equiv \pm m_1\bmod d}}1+ T^{-A}\\ \nonumber
        &\ll_{\eps,A} \frac{T^\eps X^2}{B^2} \sum_{\substack{m_1,m_2\sim M_1\\ 0\leq m_2-m_1\ll N_2\\m_1\equiv m_2\bmod d\\ }}\frac{|\lambda_g(m_1)\lambda_g(m_2)|}{m_1^{1/2}m_2^{1/2}}\sum_{\substack{n-m_1t_{f,g}\ll N_2\\ n\equiv\pm m_1\bmod d}}1+ T^{-A}\\ \nonumber
        &\ll_{\eps,A} \frac{T^\eps X^2}{B^2} \left(1+
        \frac{N_2}{d}\right)\sum_{\substack{m_1,m_2\sim M_1\\ 0\leq m_2-m_1\ll N_2\\m_1\equiv m_2\bmod d}}\frac{|\lambda_g(m_1)\lambda_g(m_2)|}{m_1^{1/2}m_2^{1/2}} +T^{-A},\
    \end{align}
    since $t_{f,g}\asymp 1$. Writing $m_2=m_1+rd$, with $r\ll N_2/d$ and estimating the $m_1,m_2$-sum using A.M-G.M inequality and Ramanujan bound on average \eqref{2.2} implies that the $m_1,m_2$- sum is 
    \begin{align}
        \ll \sum_{0\leq r\ll N_2/d}\sum_{m_1\sim M_1}\frac{|\lambda_g(m_1)|^2}{m_1}+\ \sum_{0\leq r\ll N_2/d}\sum_{m_1\sim M_1} \frac{|\lambda_g(m_1+rd)|^2}{m_1+rd}\ll_\eps T^\eps \left(1+\frac{N_2}{d}\right),
    \end{align}
    since $M_1\asymp N_0$ and $N_2=N_0T^\eps/B\ll N_0$. Therefore 
    \begin{align*}
        \Omega_d\ll_{\eps,A}  \frac{T^\eps X^2}{B}\left(1+\frac{N_2}{d}\right)^2 + T^{-A},
    \end{align*}
    and
    \begin{align*}
        S_+(N,C,X)&\ll_{\eps,A} \frac{T^\eps NX}{Q} \sup_{1\ll D\ll C}\sum_{d\sim D}\left(1+\frac{N_2}{d}\right)+T^{-A}\\
        &\ll_{\eps,A}\frac{T^\eps NX}{Q}\left(C+N_2\right)+ T^{-A}.
    \end{align*}
    With $B\ll T^{1/2+\eps}$, $C\ll Q$, $N_0=BTC^2/N$ and $N_2=N_0T^\eps/B$, we have 
    \begin{proposition}\label{prop_7.1}
        If $T^\eps\ll B=NX/CQ\ll T^{1/2-\eps}$,
        \begin{align}
            S_+(N,C,X)\ll_\eps T^\eps\left(\frac{T^{3/2}N^{1/2}}{K^{3/2}}+\frac{T^{1/2}N}{K}\right).
        \end{align}
    \end{proposition}
    \subsection{Oscillation between \texorpdfstring{$T^{1/2+\eps}$}{T^{1/2}} and \texorpdfstring{$T^{1-\eps}$}{T}}\label{Estimate_beyond_square_root}
    \subsubsection{Oscillation upto \texorpdfstring{$T^{2/3-\eps}$}{T^{2/3}}}\label{upto_two_third} If $T^{1/2+\eps}\ll B\ll T^{2/3-\eps}$,we smooth out the $n$-sum by introducing a smooth compactly supported non negative weight $\varphi$ with support in $[1/2,5/2]$ and $\varphi\equiv 1$ on $[1,2]$. Opening up the absolute value square in \eqref{6.8}, using the symmetry of the $m_1,m_2$-sums and taking the $n$-sum inside 
    \begin{align}
        \Omega_d \ll \sum_{\substack{m_1,m_2\sim M_1\\0\leq m_2-m_1 \\m_1\equiv m_2\bmod d}}\frac{\lambda_g(m_1){\lambda_g(m_2)}}{m_1^{1/2-it_g}m_2^{1/2+it_g}}\sum_{n\equiv \pm m_1\bmod d}\varphi\left(\frac{n}{N_1}\right)\mathcal I(m_1,n,dl)\overline{\mathcal I(m_2,n,dl)}.
    \end{align}
     Lemma \ref{structure_of_I_sharp} \ref{item5.2_2} implies 
     \begin{align}
         \Omega_d\ll_{\eps,A} T^\eps  \sum_{\substack{m_1,m_2\sim M_1\\ 0\leq m_2-m_1\\ m_1\equiv m_2\bmod d}}\frac{\lambda_g(m_1){\lambda_g(m_2)}}{m_1^{1/2-it_g}m_2^{1/2+it_g}}\sum_{\substack{n\equiv \pm m_1\bmod d\\ n-m_1t_{f,g}\asymp N_2\\ n-m_2t_{f,g}\asymp N_2}}\varphi\left(\frac{n}{N_1}\right)\mathcal I(m_1,n,dl)\overline{\mathcal I(m_2,n,dl)}
         + T^{-A},
     \end{align}
     i.e
     \begin{align}\label{7.7}
          \Omega_d\ll_{\eps,A} T^\eps  \sum_{\substack{m_1,m_2\sim M_1\\ m_1\equiv m_2\bmod d\\0\leq m_2-m_1\ll N_2}}\frac{\lambda_g(m_1){\lambda_g(m_2)}}{m_1^{1/2-it_g}m_2^{1/2+it_g}}\sum_{\substack{n\equiv \pm m_1\bmod d\\ n-m_1t_{f,g}\asymp N_2\\ n-m_2t_{f,g}\asymp N_2}}\varphi\left(\frac{n}{N_1}\right)\mathcal I(m_1,n,dl)\overline{\mathcal I(m_2,n,dl)}
         + T^{-A}.
     \end{align}
    From Lemma \ref{structure_of_I} \ref{item5.1_2}
    \begin{align*}
        \mathcal{I}(m_i,n,dl)= X\int_{u_i\ll T^\varepsilon B^{-1}}&\int W(x_i)e\left(-\frac{Bu_i}{2\pi}\log\frac{x_iXdl}{2\pi^2m_1Q}\right)\mathcal{I}_{u_i}^\sharp(m_i,n)\mathop{dx_i}\mathop{du_i},
    \end{align*}
    Substituting the above expression in \eqref{7.7} and taking the $n$-sum inside the $u$, $x$-integrals leads to 
    \begin{align}\label{7.8}
        \Omega_d\ll_{\eps, A} T^\eps X^2&\sum_{\substack{m_1,m_2\sim M_1\\ m_1\equiv m_2\bmod d\\0\leq m_2-m_1\ll N_2}}\frac{\lambda_g(m_1){\lambda_g(m_2)}}{m_1^{1/2-it_g}m_2^{1/2+it_g}}\int\int_{u_1,u_2\ll T^\eps B^{-1}}\int\int_{x_1,x_2}(\cdots)\\ \nonumber
        &\left(\sum_{\substack{n\equiv \pm m_1\bmod d\\ n-m_1t_{f,g}\asymp N_2\\n-m_2t_{f,g}\asymp N_2}} \varphi\left(\frac{n}{N_1}\right)\mathcal{I}_{u_1}^\sharp(m_1,n)\overline{\mathcal{I}_{u_2}^\sharp(m_2,n)}\right)\mathop{du_1du_2dx_1dx_2} + T^{-A}.
    \end{align}
    
    Substituting the expression for $\mathcal{I}_u^\sharp(m,n)$ from Lemma \ref{structure_of_I_sharp} \ref{item5.2_2}, the $n$-sum in \eqref{7.8} reduces to 
    \begin{align}\label{7.9}
        \frac{T}{B^2}\sum_{n\equiv \pm m_1\bmod d}\varphi\left(\frac{n}{N_1}\right)&V_{\eps,u_1}\left(\frac{n-m_1t_{f,g}}{N_2}\right)\overline{V_{\eps,u_2}\left(\frac{n-m_2t_{f,g}}{N_2}\right)}\\ \nonumber 
        &e\left(\frac{1}{2\pi}\left(2t\left(\log(n+m_1)-\log(n+m_2)\right)-2t_g(\log m_1-\log m_2)\right) \right).
    \end{align}
    Note that the support conditions on $n$ is already captured by the weights
    \begin{align*}
        V_{\eps,u_1}\left(\frac{n-m_1t_{f,g}}{N_2}\right)\overline{V_{\eps,u_2}\left(\frac{n-m_2t_{f,g}}{N_2}\right)}.
    \end{align*}
    Let $\Omega_d(m_2=m_1)$ and $\Omega_d(m_2\neq m_1)$ be the contributions of RHS in \eqref{7.8} when $m_2=m_1=m$ and $m_2\neq m_1$ respectively. If $m_1=m_2$ we estimate the $n$-sum in \eqref{7.9} and the $u_i,x_i$-integrals trivially, obtaining
    \begin{align}\label{7.10}
        \Omega_d(m_2=m_1)&\ll_{\eps,A} \frac{T^{1+\eps}X^2}{B^4}\left(1+\frac{N_2}{d}\right)\sum_{m\sim M_1}\frac{|\lambda_g(m)|^2}{m}+T^{-A}\\ \nonumber
        &\ll_{\eps,A} \frac{T^{1+\eps}X^2}{B^4}\left(1+\frac{N_2}{d}\right)+ T^{-A},
    \end{align}
    by Ramanujan bound on average \eqref{2.2} on the $m$-sum.
    
    If $m_2\neq m_1$, writing $n=\pm m_1+rd$, we apply Poisson summation formula on $r$ to obtain that \eqref{7.9} equals
    \begin{align}\label{7.11}
        &\frac{T}{2\pi i B^2}e\left(-2t_g(\log m_1-\log m_2\right)\\ \nonumber
        &\sum_{r\in \mathbb{Z}}\int\varphi\left(\frac{\pm m_1+wd}{N_1}\right)  V_{\eps,u_1}\left(\frac{\pm m_1+wd-m_1t_{f,g}}{N_2}\right)\overline{V_{\eps,u_2}\left(\frac{\pm m_1+wd-m_2t_{f,g}}{N_2}\right)}\\ \nonumber
        &\;\;\;\;\;\;\;\;\;\;\;\;e\left(\frac{t}{\pi}\left(\log(\pm m_1+wd+m_1)-\log(\pm m_1+wd+m_2)\right)-rw\right)\mathop{dw}.
    \end{align}
    Notice that in here the condition $d\ll N_2$ is implicit.
    Let $I_w$ denote the $w$-integral in \eqref{7.11} above. With the standard change of variable 
    \begin{align*}
        \frac{\pm m_1+wd-m_1t_{f,g}}{N_2} \rightsquigarrow w,
    \end{align*}
    implies $I_w$ equals
    \begin{align}\label{}
       \frac{N_2}{d}& e\left(-r\left(\frac{m_1t_{f,g}\mp m_1}{d}\right)+\frac{t}{\pi}\log\left(\frac{m_1+m_1t_{f,g}}{m_2+m_1t_{f,g}}\right)\right)\\ \nonumber
       &\int \Phi(w)
        e\left(\frac{t}{\pi}\left(\log\left(1+\frac{N_2w}{m_1+m_1t_{f,g}}\right)-\log\left(1+\frac{N_2w}{m_2+m_1t_{f,g}}\right)\right)-\frac{rN_2w}{d}\right)\mathop{dw},
    \end{align}
    where
    \begin{align*}
        \Phi(w)= \varphi\left(\frac{N_2w+m_1t_{f,g}}{N_1}\right)V_{\eps,u_1}\left(w\right)\overline{V_{\eps,u_2}\left(\frac{N_2w-(m_2-m_1)t_{f,g}}{N_2}\right)}.
    \end{align*}
    Note that $\Phi$ is compactly supported and $T^\eps$- inert. Let (temporarily)
    \begin{align}\label{7.13}
        h(w)&= \frac{t}{\pi}\left(\log\left(1+\frac{N_2w}{m_1+m_1t_{f,g}}\right)-\log\left(1+\frac{N_2w}{m_2+m_1t_{f,g}}\right)\right)-\frac{rN_2w}{d}\\ \nonumber
        &= \frac{tN_2(m_2-m_1)w}{\pi(m_1+m_1t_{f,g})(m_2+m_1t_{f,g})}-\frac{rN_2w}{d}+\tilde h(w),
    \end{align}
    where 
    \begin{align*}
        \tilde h(w)= \frac{t}{\pi}\sum_{j=1}^\infty \frac{(-1)^j}{j+1}\left(\frac{(N_2w)^{j+1}}{(m_1+m_1t_{f,g})^{j+1}}-\frac{(N_2w)^{j+1}}{(m_2+m_1t_{f,g})^{j+1}}\right),
    \end{align*}
    by Taylor series expansions.
    Notice that 
    \begin{align*}
       \frac{d^j}{dw^j} \tilde h(w)\ll_{j} T^\eps.
    \end{align*}
    Therefore 
    \begin{align*}
        h'(w)= \frac{tN_2(m_2-m_1)}{\pi(m_1+m_2t_{f,g})(m_2+m_1t_{f,g})}- \frac{rN_2}{d}+O(T^\eps)
    \end{align*}
    and $h^{(j)}(w)\ll_j T^\eps$ for $j\geq 2$. Again Lemma \ref{repeated_integration_by_parts} with $X_0=1$, $V_0=T^{-\eps}$, $R_0=T^\eps$, $Y_0=T^\eps$, $Q_0=1$, implies that $I_w$ is negligibly small unless
    \begin{align}\label{7.14}
        \frac{tN_2(m_2-m_1)}{\pi(m_1+m_1t_{f,g})(m_2+m_1t_{f,g})}-\frac{rN_2}{d}\ll T^\eps.
    \end{align}
    Thus, estimating the $u_i,x_i$-integrals trivially we have  
    \begin{align}\label{7.15}
        \Omega_d(m_1\neq m_2)\ll_{\eps,A} \frac{T^{1+\eps} X^2 N_2}{dB^4} \sum_{\substack{m_1,m_2\sim M_1\\ r\in \mathbb{Z}\\ m_2\equiv m_1\bmod d\\ 0<m_2-m_1\ll N_2\\ \frac{td(m_2-m_1)}{\pi(m_1+m_1t_{f,g})(m_2+m_1t_{f,g})}-r\ll \frac{dT^\eps}{N_2}}}\frac{|\lambda_g(m_1)||\lambda_g(m_2)|}{(m_1m_2)^{1/2}}+ T^{-A}.
    \end{align}
    Let $\Omega_d^{r=0}(m_1\neq m_2)$, $\Omega_d^{r\neq 0}(m_1\neq m_2)$ be the contributions from the terms corresponding to $r=0$ and $r\neq 0$ respectively.\\
    
    If $r=0$ then \eqref{7.14} implies $$m_2-m_1\ll \frac{N_0^2T^\eps}{TN_2}\ll \frac{N_0T^\eps}{B}.$$
    Hence,
    \begin{align*}
       \Omega_d^{r=0}(m_1\neq m_2)\ll_{\eps, A} \frac{T^{1+\eps}X^2 N_2}{dB^4}\sum_{\substack{m_1,m_2\sim M_1\\ m_2\equiv m_1\bmod d\\ 0<m_2-m_1\ll \frac{N_0T^\eps}{B}}}\frac{|\lambda_g(m_1)||\lambda_g(m_2)|}{(m_1m_2)^{1/2}}+ T^{-A}.
    \end{align*}
    Estimating the $m_1,m_2$-sums as earlier, using A.M-G.M inequality and Ramanujan bound on average \eqref{2.2} implies 
    \begin{align}\label{7.16}
        \Omega_d^{r=0}(m_1\neq m_2) \ll_{\eps,A} \frac{T^{1+\eps}X^2N_2N_0}{d^2B^5}+T^{-A}.
    \end{align}
    If $r\neq 0$ then \eqref{7.14} and the fact that $m_2-m_1\ll N_2$ implies 
    \[1\leq r\ll \frac{BDT^\eps}{N_1}\]
    Let $m_2=m_1+dh$ with $1\leq h\ll N_2/d$. Then taking a dyadic subdivision over $h$ and $r$, we have
    \begin{align}\label{7.17}
        \Omega_d^{r\neq 0}(m_1\neq m_2)&\ll_{\eps,A} \frac{T^{1+\eps}X^2 N_2}{dB^4}\sup_{\substack{1\ll R\ll 
        \frac{BDT^\eps}{N_1}\\ 1\ll H\ll \frac{N_2}{D}}}\sum_{\substack{r\sim R\\ h\sim H}} \\ \nonumber
        &\sum_{\substack{m_1\asymp N_0 \\ \frac{td^2h}{\pi(m_1+m_1t_{f,g})(dh+m_1+m_1t_{f,g})}-r\ll \frac{DT^\eps}{N_2}}}\frac{|\lambda_g(m_1)\lambda_g(m_1+dh)|}{m_1^{1/2}(m_1+dh)^{1/2}} + T^{-A}.
    \end{align}
    Fixing $h$ and $r$, we precisely determine all such $m_1$ satisfying the condition under the second summation in \eqref{7.17}. Let $y=m_1+m_1t_{f,g}$. Then $y\asymp N_0$ and 
    \begin{align*}
        \frac{td^2h}{\pi y(y+dh)}-r\ll \frac{DT^\eps}{N_2}.
    \end{align*}
    i.e
    \begin{align}\label{7.18}
        &\pi ry(y+dh)- td^2h\ll \frac{N_0^2DT^\eps}{N_2}\\ \nonumber
        \implies &\prod_{\pm}(y-\alpha_\pm(r,h,d))\ll \frac{N_0^2DT^\eps}{RN_2},
    \end{align}
    \begin{align}\label{7.19}
        \alpha_\pm(r,h,d)=\frac{-\pi rdh\pm \sqrt{\pi^2r^2d^2h^2+4\pi td^2hr}}{2\pi r}.
    \end{align}
    Since $\alpha_-(r,h,d)<0$, and $y\asymp N_0$, \eqref{7.18} implies that 
    \begin{align*}
        y-\alpha_+(r,h,d)\ll \frac{N_0DT^\eps}{RN_2},
    \end{align*}
    i.e
    \begin{align}\label{7.20}
        m_1-\frac{\alpha_+ (r,h,d)}{1+t_{f,g}}\ll \frac{N_0DT^\eps}{RN_2}.
    \end{align}
    With this condition, \eqref{7.17} yields
    \begin{align}\label{7.21}
        \Omega_d^{r\neq 0}(m_1\neq m_2)&\ll_{\eps, A}\frac{T^{1+\eps}X^2N_2}{dB^4}\sup_{\substack{1\ll R\ll 
        \frac{BDT^\eps}{N_1}\\ 1\ll H\ll \frac{N_2}{D}}}\sum_{\substack{r\sim R\\ h\sim H}}\\ \nonumber
        &\sum_{\substack{m\asymp N_0\\ m-\frac{\alpha_+(r,h,d)}{1+t_{f,g}}\ll \frac{N_0DT^\eps}{RN_2}}}\frac{|\lambda_g(m)||\lambda_g(m+dh)|}{m^{1/2}(m+dh)^{1/2}} + T^{-A}.
    \end{align}
    Let 
    \begin{align}\label{7.22}
        S_{R,H}= \sum_{\substack{r\sim R\\ h\sim H}}\sum_{\substack{m\asymp N_0\\ m-\frac{\alpha_+(h,r,d)}{1+t_{f,g}}\ll \frac{N_0DT^\eps}{RN_2}}}\frac{|\lambda_g(m)||\lambda_g(m+dh)|}{m^{1/2}(m+dh)^{1/2}}.
    \end{align}
    By A.M-G.M inequality 
    \begin{align}\label{7.23}
        S_{R,H}\ll S_1+S_2.
    \end{align}
    where 
    \begin{align*}
        S_1=\sum_{\substack{r\sim R\\ h\sim H}}\sum_{\substack{m\asymp N_0\\ m-\frac{\alpha_+(h,r,d)}{1+t_{f,g}}\ll \frac{N_0DT^\eps}{RN_2}}} \frac{|\lambda_g(m)|^2}{m}
    \end{align*}
    and 
    \begin{align*}
        S_2=\sum_{\substack{r\sim R\\ h\sim H}}\sum_{\substack{m\asymp N_0\\ m-\frac{\alpha_+(h,r,d)}{1+t_{f,g}}\ll \frac{N_0DT^\eps}{RN_2}}} \frac{|\lambda_g(m+dh)|^2}{m+dh}.
    \end{align*}
     Notice that $dh\ll N_2\ll N_0DT^\eps/RN_2\ll N_0$ since $B\ll T^{2/3-\eps}$, and therefore
     \[m-\frac{\alpha_+(h,r,d)}{1+t_{f,g}}\ll \frac{N_0DT^\eps}{RN_2}\iff m+dh-\frac{\alpha_+(h,r,d)}{1+t_{f,g}}\ll \frac{N_0DT^\eps}{RN_2},\]
     
     \[\implies S_2\ll S_1.\] So it is enough for us to estimate $S_1$, which we carry out in the following Lemma.
     \begin{Lemma}\label{estimate_of_S_1}
         We have 
         \begin{align}
             S_1\ll_\eps T^\eps
         \end{align}
         and thus $S_{R,H}\ll_{\eps} T^\eps$.
     \end{Lemma}
     \begin{proof}
         We first define the weights $\omega(m)$ for each $m\asymp N_0$ as follows,
         \begin{align}
             \omega(m)= \sum_{\substack{ r\sim R\\ h\sim H\\
             \frac{\alpha_+(h,r,d)}{1+t_{f,g}}-m\ll \frac{N_0DT^\eps}{RN_2}}}1,
         \end{align}
         so that
         \begin{align}
             S_1= \sum_{m\asymp N_0}\frac{|\lambda_g(m)|^2}{m}\omega(m).
         \end{align}
         Cauchy's inequality followed by Lemma \ref{L^4_norm_Fourier_Coeff} combined with partial summation implies,
         \begin{align}\label{7.27}
             S_1\ll_\eps \frac{T^\eps}{N_0^{1/2}}\left(\sum_{m\asymp N_0}\omega(m)^2\right)^{1/2}.
         \end{align}
        Now, 
        \begin{align}
         \sum_{m\asymp N_0}\omega(m)^2&=\sum_{m\asymp N_0}\left(\sum_{\substack{ r\sim R\\ h\sim H\\
             \frac{\alpha_+(h,r,d)}{1+t_{f,g}}-m\ll \frac{N_0DT^\eps}{RN_2}}}1\right)^2\\ \nonumber
             &= \sum_{m\asymp N_0}\sum_{\substack{ r_1\sim R\\ h_1\sim H\\
             \frac{\alpha_+(h_1,r_1,d)}{1+t_{f,g}}-m\ll \frac{N_0DT^\eps}{RN_2}}}\sum_{\substack{ r_2\sim R\\ h_2\sim H\\
             \frac{\alpha_+(h_2,r_2,d)}{1+t_{f,g}}-m\ll \frac{N_0DT^\eps}{RN_2}}}1 \\ \nonumber
             &= \sum_{\substack{r_1,r_2\sim R\\ h_1,h_2\sim H\\
             \alpha_+(h_1,r_1,d)-\alpha_+(h_2,r_2,d)\ll\frac{N_0DT^\eps}{RN_2}}}\sum_{\frac{\alpha_+(h_1,r_1,d)}{1+t_{f,g}} -m\ll \frac{N_0DT^\eps}{RN_2}}1\\ \nonumber
             &\ll \frac{N_0DT^\eps}{RN_2}\sum_{\substack{r_1,r_2\sim R\\ h_1,h_2\sim H\\
             \alpha_+(h_1,r_1,d)-\alpha_+(h_2,r_2,d)\ll\frac{N_0DT^\eps}{RN_2}}}1.
        \end{align}
        By a routine argument involving the Taylor series expansion of $\alpha_+(h,r,d)$,
        \begin{align}
          \alpha_+(h_1,r_1,d)-\alpha_+(h_2,r_2,d)\asymp dt^{1/2}\left(\frac{h_1^{1/2}}{r_1^{1/2}}-\frac{h_2^{1/2}}{r_2^{1/2}}\right)+O(N_2),
      \end{align}
      Therefore 
      \begin{align}
          \alpha_+(h_1,r_1,d)-\alpha_+(h_2,r_2,d)\ll \frac{N_0DT^\eps}{RN_2}\iff \frac{h_1^{1/2}}{r_1^{1/2}}-\frac{h_2^{1/2}}{r^{1/2}}\ll \frac{N_0T^\eps}{RN_2T^{1/2}}.
      \end{align}
      i.e 
      \begin{align}\label{7.31}
          h_1r_2-h_2r_1\ll \frac{T^\eps N_0H^{1/2}R^{1/2}}{N_2T^{1/2}}\ll T^\eps.
      \end{align}
      With $h_1r_2=u_1,\; h_2r_1=u_2$, The number of possible solutions to the inequality \eqref{7.31} is bounded by
      \begin{align*}
          \sum_{\substack{u_2\sim HR\\ u_1-u_2\ll T^\eps}} d(u_1)\ll T^\eps HR ,
      \end{align*}
      $d(.)$ being the divisor function. Hence 
      \begin{align*}
          \sum_{m\asymp N_0}\omega(m)^2\ll \frac{N_0DHRT^\eps}{RN_2}\ll N_0T^{\eps},
      \end{align*}
      since $DH\ll N_2$. Therefore the Lemma follows from \eqref{7.27}. 
      \end{proof}
      Using the above Lemma we obtain
      \begin{align}\label{7.32}
          \Omega_d^{r\neq 0}(m_1\neq m_2)\ll_{\eps,A} \frac{T^{1+\eps}X^2N_2}{dB^4}+T^{-A}. 
      \end{align} 
     Combining the bounds in \eqref{7.10}, \eqref{7.16}, \eqref{7.32} we obtain,
      \begin{align}
          \Omega_d &\ll_{\eps,A} \frac{T^{1+\eps}X^2}{B^4}\left(1+\frac{N_2}{d}+\frac{N_2N_0}{d^2B}\right),
      \end{align}
      and thus
      \begin{align}
          S_+(N,C,X)\ll_{\eps,A} \frac{T^{1/2+\eps} NX}{BQ}\sup_{1\ll D\ll C}\sum_{d\sim D}1+\frac{N_2^{1/2}}{d^{1/2}}+\frac{N_2^{1/2}N_0^{1/2}}{dB^{1/2}}+T^{-A}\\ \nonumber
          \ll_{\eps,A}\frac{T^{1/2+\eps} NX}{BQ}\left(C+N_2^{1/2}C^{1/2}+\frac{N_2^{1/2}N_0^{1/2}}{B^{1/2}}\right) + T^{-A}.
      \end{align}
      With $X\ll T^\eps$, $C\ll Q$, $N_0=BTC^2/N$, $N_2=N_0B/T$, we have
      \begin{proposition}\label{prop_7.3}
          If $T^{1/2+\eps}\ll B=NX/CQ\ll T^{2/3-\eps}$,
          \begin{align}
              S_+(N,C,X)\ll_\eps T^\eps\left(\frac{T^{1/2}N}{K}+\frac{T^{1/2}N^{3/4}}{K^{1/4}}+\frac{TN^{1/2}}{K^{1/2}}\right).
           \end{align}
      \end{proposition}
    \subsubsection{Oscillation beyond \texorpdfstring{$T^{2/3-\eps}$}{T^{2/3}}}\label{beyond_two_third} If $T^{2/3-\eps}\ll B\ll T^{1-\eps}$ we again estimate $\Omega_d$ trivially as in \eqref{7.1}. So that, by Lemma \ref{5.2} \ref{item5.2_2}
    \begin{align}
         \Omega_d\ll_{\eps,A}\frac{T^{1+\eps} X^2}{B^4}\sum_{n\sim N_1}\left|\sum_{\substack{m\sim M_1\\ m\equiv \pm n\bmod d\\ n-mt_{f,g}\ll N_2}}\frac{\lambda_g(m)}{m^{1/2-it_g}}\right|^2 + T^{-A},
    \end{align}
    where $N_1,M_1\asymp N_0$ and $N_2=N_0B^{1+\eps}/T$. Opening up the absolute value square and estimating the $m_1,m_2$-sum exactly as in \eqref{7.2} we get 
    \begin{align}
        \Omega_d\ll_{\eps,A} \frac{T^{1+\eps}X^2}{B^4}\left(1+\frac{N_2}{d}\right)^2.
    \end{align}
    Thus
    \begin{align}
        S_+(N,C,X)&\ll_{\eps,A} \frac{T^\eps NX}{Q}\frac{T^{1/2}}{B}\sup_{1\ll D\ll C}\sum_{d\sim D}\left(1+\frac{N_2}{d}\right)+ T^{-A}\\ \nonumber
        &\ll_{\eps,A} \frac{T^{1/2+\eps} NX}{QB}\left(C+N_2\right)+ T^{-A}.
    \end{align}
    With $X\ll T^\eps$, $B\gg T^{2/3-\eps}$, $C=NX/BQ$, $N_2=N_0B/T$, $N_0= BTC^2/N$, we obtain
    \begin{proposition}\label{prop_7.4}
        If $T^{2/3+\eps}\ll B=NX/CQ \ll T^{1-\eps}$ then,
        \begin{align}
            S_+(N,C,X)\ll_\eps T^\eps\left(\frac{NK}{T^{5/6}}+ \frac{N^{1/2}K^{3/2}}{T^{1/6}}\right).
        \end{align}
    \end{proposition}
    \subsection{Oscillation beyond \texorpdfstring{$T^{1-\eps}$}{T}} If $B\gg T^{1-\eps},$ we smooth out the $n$-sum by introducing a smooth compactly supported weight $\varphi$ with support in $[1/2,5/2]$ and $\varphi\equiv 1$ on $[1,2]$. Opening the absolute value square, and using the symmetry of the $m_1,m_2$-sums we get
    \begin{align}\label{7.41}
        \Omega_d\ll \sum_{\substack{m_1,m_2\sim M_1 \\0\leq m_2-m_1 \\m_1\equiv m_2 \bmod d}} \frac{\lambda_g(m_1){\lambda_g(m_2)}}{m_1^{1/2-it_g}m_2^{1/2+it_g}}\sum_{n\equiv \pm m_1\bmod d}\varphi\left(\frac{n}{N_1}\right)\mathcal I(m_1,n,dl)\overline{\mathcal I(m_2,n,dl)}.
    \end{align}
    Writing $n=\pm m_1+rd$ and applying Poisson summation formula on the $r$-sum implies that the $n$-sum equals
    \begin{align}\label{7.42}
        &=\frac{1}{2\pi i}\sum_{r}\varphi\left(\frac{\pm m_1+dr}{N_1}\right) \mathcal I(m_1,m_1+dr,dl)\overline{\mathcal I(m_2,m_1+dr,q)}\\ \nonumber
        &\asymp  \sum_{r}\int\phi\left(\frac{\pm m_1+dw}{{N_1}}\right) \mathcal I(m_1,m_1+dw,dl)\overline{\mathcal I(m_2,m_1+dw,dl)}e(-rw)\mathop{dw}\\ \nonumber
        &\asymp \frac{N_1}{d}\sum_{r}e\left(\frac{\pm rm_1}{d}\right)\int \varphi(w)\mathcal{I}(m_1,N_1w,dl)\overline{\mathcal{I}(m_2,N_1w,dl)}e\left(-\frac{r{N_1}w}{d}\right)\mathop{dw},
    \end{align}
    by a standard change of variables.
    Next, from Lemma \ref{structure_of_I} \ref{item5.1_2} we have, \ref{item5.1_2_2}
    \begin{align*}
        \mathcal{I}(m_i,N_1w,dl)= X\int_{u_i\ll T^\varepsilon B^{-1}}&\int W(x_i)e\left(-\frac{Bu_i}{2\pi}\log\frac{x_iXdl}{2\pi^2m_iQ}\right)\mathcal{I}_{u_i}^\sharp(m_i,N_1w)\mathop{dx_i}\mathop{du_i}
    \end{align*}
   
    Then the multiple integral in \eqref{7.42} boils down to 
    \begin{align}\label{7.43}
        X^2\int\int_{u_1,u_2\ll T^{\eps}B^{-1}}&\int\int_{x_1,x_2}(\cdots)\\ \nonumber
        &\int \varphi(w)\mathcal{I}_{u_1}^\sharp(m_1,N_1w)\overline{\mathcal{I}_{u_2}^\sharp(m_2,N_1w)}e\left(-\frac{r{N_1}w}{d}\right)\mathop{dw}\mathop{dx_1}\mathop{dx_2}\mathop{du_1}\mathop{du_2}
    \end{align}
    Isolating the $w$ integral briefly by using Lemma \ref{structure_of_I} \ref{item5.1_2} \ref{item5.1_2_2} once more, it equals
    \begin{align*} 
       I_w= \int \varphi(w)e \left(\frac{B}{2\pi}(\tau_1-\tau_2)\log w\right)e\left(-\frac{r{N_1}w}{d}\right)dw.
    \end{align*}
    Repeated integration by parts implies that the $I_w$ is negligibly small unless $r\ll BdT^\eps/N_1,$ in which case we obtain
    by Cauchy's inequality and Lemma \ref{structure_of_I_sharp} \ref{item5.2_3} (after estimating the $u_i,x_i$-integrals trivially) that \eqref{7.43} is 
    \begin{align}
        \ll_\eps \frac{T^\eps X^2}{B^3}.
    \end{align}
    Then \eqref{7.42} is 
    \begin{align}
        \ll_{\eps,A} \frac{ T^\eps N_1X^2}{dB^3}\sum_{r\ll \frac{BdT^\eps}{N_1}} 1\ll_{\eps,A}  \frac{ T^\eps N_1X^2}{dB^3}\left(1+\frac{Bd}{N_1}\right) + T^{-A}, 
    \end{align}
    whence
    \begin{align*}
        \Omega_d\ll_{\eps, A}\frac{ T^\eps N_1X^2}{dB^3}\left(1+\frac{Bd}{N_1}\right)  \sum_{\substack{m_1,m_2\sim M_1\\0\leq m_2-m_1 \\ m_1\equiv m_2 \bmod d}} \frac{|\lambda_g(m_1)\lambda_g(m_2)|}{(m_1m_2)^{1/2}}.
    \end{align*}
    Estimating the $m_1,m_2$-sum using the AM-GM inequality and Ramanujan bound on average \eqref{2.2} as earlier yields,
    \begin{align*}
        \Omega_d &\ll_{\eps, A}\frac{ T^\eps N_1X^2}{dB^3}\left(1+\frac{Bd}{N_1}\right) \left(1+\frac{M_1}{d}\right)+ T^{-A}\\ &\ll_{\eps,A} \frac{T^\eps X^2}{B^2}\left(\frac{N_1}{dB}+1\right)\left(1+\frac{M_1}{d}\right) +T^{-A}\\ 
        &\ll_{\eps,A} \frac{T^\eps X^2}{B^2}\left(\frac{N_0}{dB}+1\right)\left(1+\frac{N_0}{d}\right) +T^{-A}\\
        &\ll_{\eps, A} \frac{T^\eps X^2}{B^2}\left(1+\frac{N_0}{d}+\frac{N_0^2}{Bd^2}\right)+ T^{-A},
    \end{align*}
   after taking supremum over $N_1,M_1\ll N_0=M_0$. Thus 
   \begin{align}
       S_+(N,C,X)&\ll_{\eps, A} \frac{NX}{Q}\sup_{1\ll D\ll C}\sum_{d\sim D} \left(1+\frac{N_0}{d}+\frac{N_0^2}{Bd^2}\right)^{1/2}+ T^{-A}\\ \nonumber 
       &\ll _{\eps, A} \frac{NX}{Q}\sup_{1\ll D\ll C}\sum_{d\sim C} \left(1+\frac{N_0^{1/2}}{d^{1/2}}+\frac{N_0}{B^{1/2}d}\right)+ T^{-A}\\ \nonumber
       &\ll_{\eps,A} \frac{NX}{Q}\left(C+N_0^{1/2}C^{1/2}+\frac{N_0}{B^{1/2}}\right) + T^{-A}.
   \end{align}
   With $B\gg T^{1-\eps}$, $X\ll T^\eps$ $C=NX/BQ$, $N_0=B^2C^2T^\eps/N\ll_\eps NT^\eps/Q^2\ll_\eps KT^\eps$, we have 
   \begin{proposition}\label{prop_7.5}
       If $B=NX/CQ\gg T^{1-\eps}$,
       \begin{align}
           S_+(N,C,X)\ll_\eps T^\eps \left(\frac{NK}{T}+ \frac{N^{3/4}K^{5/4}}{T^{1/2}}+\frac{N^{1/2}K^{3/2}}{T^{1/2}}\right).
       \end{align}
   \end{proposition}
   \section{Proof of Theorem 1.1}
   We are now in a position to prove Theorem \ref{Theorem}. Putting together the bounds from  Propositions \ref{prop_6.1}, \ref{prop_7.1}, \ref{prop_7.3}, \ref{prop_7.4}, \ref{prop_7.5} in \eqref{2.28}, we get
   \begin{align}
       S(N) \ll_{\eps} T^\eps\Biggl(&\frac{N}{K}+\frac{TN^{1/2}}{K^{3/2}}+\frac{T^{3/2}N^{1/2}}{K^{3/2}}+\frac{T^{1/2}N}{K}+\frac{T^{1/2}N^{3/4}}{K^{1/4}}\\ \nonumber&+\frac{NK}{T^{5/6}}+ \frac{N^{1/2}K^{3/2}}{T^{1/6}}+\frac{NK}{T}+\frac{N^{3/4}K^{5/4}}{T^{1/2}}+\frac{N^{1/2}K^{3/2}}{T^{1/2}}\Biggr).
   \end{align}
   With $N\ll T^{3/2+\eps}$ from \eqref{2.28},
   \begin{align}\label{8.2}
       \frac{S(N)}{\sqrt{N}}\ll_\eps T^\eps\Biggl(&\frac{T^{3/2}}{K^{3/2}}+\frac{T^{5/4}}{K}+\frac{T^{7/8}}{K^{1/4}}\\ \nonumber&+\frac{K}{T^{1/12}}+ \frac{K^{3/2}}{T^{1/6}}+\frac{K^{5/4}}{T^{1/8}}\Biggr).
   \end{align}
   Imposing the restriction (necessary for cancellation in all the contributing terms) $$T^{1/2}< K< T^{11/18},$$ it is routine to observe from \eqref{8.2} that
   \begin{align}\label{8.3}
       \frac{S(N)}{\sqrt{N}}\ll_{\eps} \frac{T^{7/8}}{K^{1/4}}+\frac{K^{3/2}}{T^{1/6}}.
   \end{align}
   The optimal choice of $K$ is $K=T^{25/42}$. Plugging this value of $K$ in \eqref{8.3} and choosing $\theta=1/42$ in Lemma  \ref{Approx_func_eq} we obtain
   \begin{align}
       L\left(1/2+it,f\otimes g\right)\ll_{\eps} T^{61/84 +\eps},
   \end{align}
   thereby proving Theorem \ref{Theorem}.
   \section{Modifications for Theorem 1.2}\label{Modifictions}
    As earlier the approximate functional equation reduces the problem of obtaining subconvexity to obtaining cancellations in the weighted Dirichlet seies
    \begin{align}
        S(N)= \sum_{n}\lambda_f(n)\lambda_g(n)n^{-it} V(n/N),
    \end{align}
    for $T^{1+\nu/2-\theta}< N<T^{1+\nu/2+\eps}.$
    Writing $t=t_f-t_g,$ we separate oscillations using the delta symbol and take smooth dyadic partitions of the $x$-integral and the $q$-sum, which leads to
    \begin{align}\label{9.2}
        S(N) \ll_{\varepsilon} T^\varepsilon\sup_{\substack{1\ll C\ll Q \\ T^{-100} \ll X\ll Q^\varepsilon\\ \pm}} |S_{\pm}(N,C,X)| + T^{-96+\varepsilon},
    \end{align}
    where 
    \begin{align}
        S_{\pm}(N,C,X) &= \frac{1}{Q}\int_{\mathbb{R}}W\left(\frac{\pm x}{X} \right)\sum_{q\sim C} \frac{g(q,x)}{q}\; \sideset{}{^*}\sum_{a\bmod q}\\
        \nonumber
		&\quad \quad \sum_{n=1}^\infty \lambda_f(n)e\left(\frac{an}{q}\right)n^{-it_f}e\left(\frac{nx}{qQ}\right)V\left(\frac{n}{N}\right)\\
        \nonumber
		&\quad\;\;\;\, \sum_{m=1}^\infty \lambda_g(m) e\left(-\frac{am}{q}\right)m^{it_g}e\left(-\frac{mx}{qQ}\right)U\left(\frac{m}{N}\right)dx.
    \end{align}
    As earlier it is enough to estimate $S_+(N,C,X)$.
    Next we apply Voronoi summation formula on the sums
    \begin{align}
        \sum_{n=1}^\infty \lambda_f(n)e\left(\frac{an}{q}\right)n^{-it_f}e\left(\frac{nx}{qQ}\right)V\left(\frac{n}{N}\right)
    \end{align} and
    \begin{align}
        \sum_{m=1}^\infty \lambda_g(m) e\left(-\frac{am}{q}\right)m^{it_g}e\left(-\frac{mx}{qQ}\right)U\left(\frac{m}{N}\right).
    \end{align}
    Next with $B=NX/CQ$ we analyse the integral transforms arising from the Voronoi summation formula exactly as in \S \ref{Prem_Int_transform}. After further simplifications using the $x$-integral and the $a$-sum as in \S \ref{Ramanujan_sum} and \S \ref{simplification_of_I} the estimation of $S_+(N,C,X)$ for the cases $B\ll T^\eps$ and $B\gg T^{1-\eps}$ is exactly what we obtained in Propositions \ref{prop_6.1} and \ref{prop_7.5}, i.e 
    \begin{align}\label{9.6}
        S_+(N,C,X)\ll_\eps T^\eps\left(\frac{N}{K}+\frac{TN^{1/2}}{K^{3/2}}\right),
    \end{align}
    if $B\ll T^\eps$ and 
    \begin{align}\label{9.7}
        S_+(N,C,X)\ll_\eps T^\eps\left(\frac{NK}{T}+\frac{N^{3/4}K^{5/4}}{T^{1/2}}+\frac{N^{1/2}K^{3/2}}{T^{1/2}}\right),
    \end{align}
    if $B\gg T^{1-\eps}$.
    
    For $T^{\eps}\ll B\ll T^{1-\eps}$ we write $q=dl$ and take a dyadic subdivision over $d$ as in \eqref{6.5} yielding
    \begin{align}
         S_+(N,C,X) \ll_{\eps,A} \frac{NB}{Q}&\sum_{1\leq l\ll \frac{C}{D}}\frac{1}{l}\sum_{d\sim D}\\ \nonumber
                &\sum_{\pm}\sum_{n\asymp N_0}\frac{|\lambda_f(n)|}{n^{1/2}}\left|\sum_{\substack{m\asymp N_0\\ m\equiv \pm n \bmod d}}\frac{\lambda_g(m)}{m^{1/2+it_g}}\mathcal{I}(m,n,dl)\right|+ T^{-A}
    \end{align}
    where $N_0=BTC^2/N$ and 
    \begin{align}
        \mathcal{I}(m,n,dl)=  X\int_{u\ll T^\varepsilon B^{-1}}&\int W(x)e\left(-\frac{Bu}{2\pi}\log\frac{xXdl}{2\pi^2mQ}\right)\mathcal{I}^\sharp_u(m,n)\mathop{dx}\mathop{du},
    \end{align}
    $\mathcal{I}_u^\sharp(m,n)$ being negligibly small unless 
    \begin{align}\label{9.10}
        n-mt_{f,g}\ll N_2= \frac{N_0T^\eps}{B},
    \end{align}
    if $B\ll T^{1-\nu/2+\eps}$ and 
    \begin{align}\label{9.11}
        n-mt_{f,g}\asymp N_2= \frac{N_0B}{T^{2-\nu}},
    \end{align}
    in which case we also have 
    \begin{align}\label{9.12}
        \mathcal{I}_u^\sharp(m,n)= \frac{T}{BT^{\nu/2}}V_{\eps,u}\left(\frac{n-mt_{f,g}}{N_2}\right)e\left(\frac{1}{2\pi}H(n,m)\right) +O_A(T^{-A}),
    \end{align}
    where 
    \begin{align*}
        H(m,n)=-2(t_g-t_f)\log\frac{e(m-n)}{t_g-t_f}+2t_g\log m-2t_f\log n,
    \end{align*}
    if $T^{1-\nu/2+\eps}\ll B\ll T^{1-\eps}$. 
    Next after applying Cauchy on $n$, opening up the absolute value square and changing the order of summations, we estimate $S_+(N,C,X)$ as in \S\S \ref{Estimate_upto_square_root} and \S\S\S \ref{beyond_two_third} if $B\ll T^{1-\nu/2+\eps}$ and $B\gg T^{1-\nu/3-\eps}$ respectively, obtaining
    \begin{align}\label{9.13}
        S_+(N,C,X)\ll_\eps T^\eps\left(\frac{T^{1-\nu/2}N}{K}+\frac{T^{2-\nu/2}N^{1/2}}{K^{3/2}}\right),
    \end{align}
    if $B\ll T^{1-\nu/2+\eps}$ and
    \begin{align}\label{9.14}
        S_+(N,C,X)\ll_\eps T^{\eps}\left(\frac{NK}{T^{1-\nu/6}}+\frac{N^{1/2}K^{3/2}}{T^{1-5\nu/6}}\right),
    \end{align}
    if $T^{1-\nu/3-\eps}\ll B\ll T^{1-\eps}$.
    
    For the range $T^{1/2-\nu+\eps}\ll B\ll T^{1-\nu/2-\eps}$, we estimate the contribution from terms $m_1=m_2$ trivially as \eqref{7.10}. For the contribution from the terms $m_1\neq m_2,$ we apply Poisson Summation formula on the $n$-variable. Treating the Fourier transform as \S\S\S \ref{upto_two_third}, we estimate the zero frequency contribution as \eqref{7.16} and in the nonzero frequency we see that its contribution is governed by an averaged shifted convolution similar to $S_{RH}$ in \eqref{7.23} and we estimate this sum as Lemma \ref{estimate_of_S_1} to finally get
    \begin{align}\label{9.15}
        S_+(N,C,X)\ll_\eps T^{\eps}\left(\frac{T^{1-\nu/2}N}{K}+\frac{T^{1/2}N^{3/4}}{K^{1/4}}+\frac{TN^{1/2}}{K^{1/2}}\right).
    \end{align}
    Combining the estimates of \eqref{9.6}, \eqref{9.7}, \eqref{9.13}, \eqref{9.14} and \eqref{9.15} in \eqref{9.2} we get
    \begin{align}
        S(N)\ll_\eps T^\eps\bigg(&\frac{T^{1-\nu/2}N}{K} + \frac{NK}{T^{1-\frac{\nu}{6}}}+ \frac{N^{1/2}T}{K^{3/2}}+\frac{N^{1/2}K^{3/2}}{T^{1-5\nu/6}}\\ \nonumber &+ \frac{T^{2-\nu/2}N^{1/2}}{K^{3/2}}+\frac{T^{1/2}N^{3/4}}{K^{1/4}}+\frac{N^{3/4}K^{5/4}}{T^{1/2}}\bigg).
    \end{align}
    Then with $N\ll T^{1+\nu/2+\eps}$,
    \begin{align}\label{9.17}
        \frac{S(N)}{\sqrt{N}}\ll_\eps T^{\eps}\bigg(&\frac{T^{3/2-\nu/4}}{K}+\frac{K}{T^{1/2-5\nu/12}}+\frac{T}{K^{3/2}}+\frac{K^{3/2}}{T^{1-5\nu/6}}\\ \nonumber
        &+\frac{T^{2-\nu/2}}{K^{3/2}}+\frac{T^{3/4+\nu/8}}{K^{1/4}}+
        \frac{K^{5/4}}{T^{1/4-\nu/8}}\bigg).
    \end{align}
    Imposing the restriction (necessary for cancellation in all the contributing terms)
    \[T^{1-\nu/2}<K<\min\{T^{1-7\nu/18}, T^{3/5+\nu/10}\},\]
    reduces \eqref{9.17} to
    \begin{align}
        \frac{S(N)}{N^{1/2}}\ll_\eps T^\eps \left(\frac{T^{3/4+\nu/8}}{K^{1/4}}+\frac{K^{3/2}}{T^{1-5\nu/6}}+\frac{K^{5/4}}{T^{1/4-\nu/8}}\right).
    \end{align}
    Notice that the restriction also forces $\nu>2/3$.

    If $2/3+\eps<\nu\leq 14/17$, we choose $K=T^{2/3}$ and $\theta= \nu/8-1/12$, obtaining
    \begin{align}
        L(1/2+it,f\otimes g)\ll_\eps  T^{7/12+\nu/8+\eps},
    \end{align}
    and if $14/17\leq \nu\leq 1$, we choose $K= T^{1-17\nu/42}$ and $\theta= \nu/42$, obtaining
    \begin{align}
        L(1/2+it,f\otimes g)\ll_\eps T^{1/2+19\nu/84+\eps}.
    \end{align}
    This completes the proof of Theorem \ref{Theorem2}.
    \section{Acknowledgement}
    The author is grateful to Prof. Ritabrata Munshi for suggesting the problem, sharing his ideas and providing his continuous guidance and support throughout the course of this work. He would also like to thank Sumit Kumar, Pratim Mitra, Sampurna Pal and Mayukh Dasaratharaman for many useful discussions and conversations. The author is also grateful to Indian Statistical Institute, Kolkata for providing an excellent and productive research atmosphere.

    \end{document}